\documentclass[a4paper,reqno,11pt]{amsart}

\usepackage[left=2.5cm,right=2.5cm,top=2.4cm,bottom=2.4cm]{geometry}
\usepackage[T1]{fontenc}
\usepackage{lmodern}
\usepackage{microtype}
\usepackage{amsmath,amssymb,amsthm,mathtools}
\usepackage{mathrsfs}
\usepackage{enumitem}
\usepackage{needspace}
\usepackage[hidelinks]{hyperref}
\hypersetup{
pdftitle={Mass Deformations and Sheets in Logahoric Higgs Moduli}, pdfauthor={Georgios Kydonakis and Lutian Zhao} }
\allowdisplaybreaks
\theoremstyle{plain}
\newtheorem{theorem}{Theorem}[section]
\newtheorem{proposition}[theorem]{Proposition}
\newtheorem{lemma}[theorem]{Lemma}
\newtheorem{corollary}[theorem]{Corollary}
\newtheorem*{theoremA}{Theorem A}
\newtheorem*{theoremB}{Theorem B}
\newtheorem*{theoremC}{Theorem C}

\theoremstyle{definition}
\newtheorem{definition}[theorem]{Definition}
\newtheorem{example}[theorem]{Example}
\newtheorem{notation}[theorem]{Notation}
\theoremstyle{remark}
\newtheorem{remark}[theorem]{Remark}

\makeatletter
\let\c@equation\c@theorem
\makeatother
\numberwithin{equation}{section}

\newcommand{\C}{\mathbb C}
\newcommand{\g}{\mathfrak g}
\newcommand{\GD}{{G_D}}
\newcommand{\gD}{{\mathfrak g_D}}
\newcommand{\lalg}{\mathfrak m}
\newcommand{\palg}{\mathfrak p}
\newcommand{\nalg}{\mathfrak n}
\newcommand{\talg}{\mathfrak t}
\newcommand{\zalg}{\mathfrak z}
\newcommand{\zM}{\mathfrak z}
\newcommand{\cO}{\mathcal O}
\newcommand{\cB}{\mathcal B}
\newcommand{\cBP}{\mathcal B_{P_M}}
\newcommand{\cF}{\mathcal F}
\newcommand{\cJ}{\mathcal J}
\newcommand{\cL}{\mathcal L}
\newcommand{\cP}{\mathcal P}
\newcommand{\cU}{\mathcal U}
\newcommand{\cV}{\mathcal V}
\newcommand{\cX}{\mathcal X}
\newcommand{\Res}{\operatorname{Res}}
\newcommand{\CoRes}{\operatorname{CoRes}}
\newcommand{\Lie}{\operatorname{Lie}}
\newcommand{\Ad}{\operatorname{Ad}}
\newcommand{\ad}{\operatorname{ad}}
\newcommand{\Ind}{\operatorname{Ind}}
\newcommand{\Spec}{\operatorname{Spec}}
\newcommand{\reg}{\mathrm{reg}}
\newcommand{\der}{\mathrm{der}}
\newcommand{\sslash}{\mathbin{/\mkern-6mu/}}

\title{Mass Deformations and Sheets in Logahoric Higgs Moduli}
\author{Georgios Kydonakis}
\address{Department of Mathematics, University of Patras, University Campus,
Patras 26504, Greece}
\email{gkydonakis@math.upatras.gr}
\author{Lutian Zhao}
\address{Kavli Institute for the Physics and Mathematics of the Universe (WPI),
The University of Tokyo, Kashiwa, Chiba 277-8583, Japan}
\email{lutian.zhao@ipmu.jp}
\date{September 2026}

\usepackage{xcolor}

\begin{document}

\begin{abstract}
We use the framework of the logahoric Hitchin integrable system introduced in our earlier work and the sheets of the Lie algebra in order to describe the tame defects in a six dimensional $\mathcal{N}=(2,0)$ superconformal theory under mass deformations. We show that varying leaves of fixed dimension in the Hitchin system amounts to varying residue orbits of fixed dimension. Moreover, we verify the period and dimension conditions in mass-like deformations, as proposed by Balasubramanian--Distler, thus realizing the sheets of the Lie algebra on the Poisson moduli space of logahoric Higgs torsors.
\end{abstract}
\maketitle
\section{Introduction}\label{sec:introduction}

One of the remarkable properties of Hitchin systems is that they arise as the Coulomb branch of class $\mathcal{S}$ $4d$ $\mathcal{N}=2$ quantum field theories compactified on $\mathbb{R}^3 \times S^1$ \cite{Ga14}. A class $\mathcal{S}$ theory is specified by the data of a punctured Riemann surface, a compact simple Lie algebra of type $A$, $D$ or $E$ and certain constraints at the punctures of the surface called \emph{defects}. Such theories are obtained by compactifying the $6d$ $\mathcal{N}=(2,0)$ superconformal field theory on the Riemann surface with a partial holomorphic twist, and, indeed, class $\mathcal{S}$ theories can be uniformly described in terms of these $6d$ theories (see \cite{GaWi}). A class $\mathcal{S}$ theory corresponds to the $4d$ $\mathcal{N}=2$ theory obtained by compactifying the $6d$ theory, with codimension-2 defects inserted at the punctures. The codimension-2 defects of the theory are then exhibited as \emph{tame} or \emph{wild} singularities of the Hitchin system. 

Chacaltana, Distler and Tachikawa in \cite{CDT} studied the behavior of codimension-2 defects of $6d$ $\mathcal{N}=(2,0)$ theories via dimensional reductions. One first such reduction is by formulating the theory on a circle transverse to the defect, and this associates a nilpotent orbit $\mathcal{O}_N$ in the Lie algebra $\mathfrak{g}$, called the \emph{Nahm label}, to the defect. Another dimensional reduction is by compactification to five dimensions on a circle that the defect wraps, and so a nilpotent orbit $\mathcal{O}_H$ in the Langlands dual Lie algebra $\mathfrak{g}^{\vee}$, called the \emph{Hitchin label} can be associated to the defect \cite{GMN13}. In both cases, a five-dimensional gauge theory emerges: for group $G$ on the Nahm side and for group $G^{\vee}$ on the Hitchin side. The pair of nilpotent orbits $(\mathcal{O}_N, \mathcal{O}_H)$ is then describing the defect. 

The study of mass deformations of $T^{\rho}[\mathfrak{g}]$ theories already appears in \cite{CDT}. Balasubramanian and Distler consider in \cite{BD} a \emph{tamely ramified Hitchin system}, meaning a meromorphic Hitchin system in which the Higgs field has simple poles and the residues at these poles are fixed adjoint orbits. In their work, they  classify all mass-like deformations of this system and then identify those mass-like deformations that indeed correspond to actual mass deformations of $T^{\rho}[\mathfrak{g}]$ theories. In particular, they use \emph{sheets} of the
Hitchin Lie algebra to organize deformations of these local residue orbits.
For a complex Lie algebra  $\mathfrak{g}$ and for $U_d$ the union of adjoint orbits in $\mathfrak{g}$ of complex dimension $d$, then the irreducible components of $U_d$ are called \emph{sheets} of the Lie algebra. Sheets were first studied by Borho and Kraft \cite{Borho}, \cite{BK}; see also \cite[Section 4.1]{BD}. A sheet always contains exactly one nilpotent orbit but it may also contain non-nilpotent orbits. One calls a sheet \emph{Dixmier}, when it contains orbits that are semisimple. Sheets are parameterized by \emph{sheet labels}, that is, pairs of certain Levi subalgebras and nilpotent orbits in the Levi subalgebras. In \cite{BD}, the Kirillov--Kostant--Souriau (KKS) holomorphic symplectic form $\Omega_{\text{KKS}}$ on the symplectic leaves of the Hitchin system varies linearly with the hypermultiplet mass parameters. These mass-like deformations are labeled by sheet labels. 

\medskip 

The aim of this paper is to place the study of mass-deformations in the framework of the \emph{logahoric Hitchin integrable system} that was formally introduced in our previous work \cite{KZ}. This was established as an algebraically completely integrable Hamiltonian system over a smooth Riemann surface with a finite number of distinct punctures, for a general complex reductive Lie algebra, and involves a (tame) moduli space of parahoric torsors over the surface with Higgs fields having logarithmic singularities at the punctures. We use this new mathematical framework in order to describe the \emph{tame defects} of the $6d$ theory under mass deformations, considering at the same time the approach of Balasubramanian and Distler \cite{BD} who used the sheets of the Lie algebra to study these tame defects.
We realize how the symplectic geometry of the logahoric Hitchin system varies with the residues by constructing families through cotangent reduction and computing their symplectic periods.
A logarithmic Higgs field is locally a Lie algebra valued one-form $A(z)\,dz/z$, with residue $A(0)$. Changing the trivialization conjugates this residue, so its orbit is intrinsic, and these orbits determine the symplectic leaves in the Poisson geometry of the moduli space of logahoric Higgs torsors. A more detailed account of our main results is provided below.

\vspace{3mm}

Let $X$ be a smooth complex projective curve with one marked point, $G$ a connected complex semisimple group, and $\mathcal G_\theta$ a parahoric group scheme with reductive quotient $L_\theta$. Let $\cB$ denote a fixed nonempty connected component of the moduli space of regularly stable parahoric torsors $\mathcal G_\theta$ over $X$; the automorphism group of these torsors is the center $Z(G)$. A level structure can be seen as a choice of framing for the $L_\theta$-torsor at the marked point. Modulo these central automorphisms, levels therefore form a principal $\GD=L_\theta/Z(G)$-bundle $\cU\to\cB$, and $T^*\cU$ parametrizes Higgs fields with their levels.

Let now $\gD=\Lie(\GD)$ denote the Lie algebra of this level group $G_D$. An invariant nondegenerate form $B_{\gD}$ converts a residue $y\in\gD$ into its \emph{coresidue} $B_{\gD}(y,-)$. This coresidue then provides  a cotangent moment map $\mu:T^*\cU\to\gD^*$ (Theorem~\ref{rem:momentlocal}), as established in \cite{KZ}. Since the Poisson moduli space $\cP=T^*\cU/\GD$ forgets the level, fixing the residue orbit $\mathscr O$ gives $\cL_{\mathscr O}=\mu^{-1}(\mathscr O)/\GD$, with adjoint and coadjoint orbits identified by $B_{\gD}$ and we show that cotangent reduction provides precisely the leaves of the Poisson foliation and their dimensions:

\begin{theoremA}[Theorem~\ref{thm:allleaves}]
Every $\cL_{\mathscr O}$ is nonempty, smooth, and connected, of dimension $2\dim\cB+\dim\mathscr O$. These reductions are exactly the symplectic leaves of the Poisson space $\cP$.
\end{theoremA}

Thus varying leaves of a fixed dimension amounts to varying residue orbits of a fixed dimension. Balasubramanian--Distler identify this dimension condition and linear symplectic periods in mass-like deformations of tame Hitchin systems \cite[Section~2.4.1]{BD}. A local example in $\mathfrak{sl}_2$ shows why retaining a parabolic reduction helps at the nilpotent limit. A residue preserving a chosen line has the form
\[
 y(a,b)=\begin{pmatrix}a&b\\0&-a\end{pmatrix}.
\]
For $a\ne0$, it is conjugate to $\operatorname{diag}(a,-a)$; for $a=0$ and $b\ne0$, it is nonzero nilpotent. We observe that both orbits have dimension two. Retaining the line includes $y=0$ in a smooth space: at $a=0$, compatible pairs satisfy $\operatorname{im}y\subset\ell\subset\ker y$, so the operators over $\ell$ form $\operatorname{Hom}(\C^2/\ell,\ell)=T^*_\ell\mathbb P^1$. The zero fiber is therefore $T^*\mathbb P^1$, whose complement of the zero section maps isomorphically to the nonzero nilpotent orbit.

In higher rank, the data consist of a parabolic $P_M=MN\subset\GD$, with Levi factor $M$ and unipotent radical $N$, its Levi algebra $\lalg=\Lie(M)$ and center $\zM=\zalg(\lalg)$, and a fixed nilpotent $M$-orbit $\cO\subset[\lalg,\lalg]$. A parameter $m\in\zM$ prescribes the Levi residue $m+e$, $e\in\cO$. When the centralizer of $m$ is $M$, the stabilizer of $m+e$ is the stabilizer of $e$ in $M$; hence its orbit dimension stays constant as $m$ varies. At zero, Lusztig--Spaltenstein induction \cite{LS} supplies an orbit of the same dimension: $\mathscr O_0=\Ind_{\lalg}^{\gD}(\cO)$ meets $\cO+\Lie(N)$ in an open dense subset. Retaining the parabolic reduction connects these orbits through the generalized Springer construction studied by Borho--MacPherson \cite{BorhoMacPherson1983} and Broer \cite{Broer}, and its deformations parametrized by the Levi center developed by Namikawa \cite{Namikawa2008,Namikawa2009,Namikawa2022CoversII}.

For Higgs moduli, the quotient spaces are $\cBP=\cU/P_M$ and $\cV=\cU/N$, so $\cV\to\cBP$ is an $M$-bundle. Since $\Lie(N)^\perp=\Lie(P_M)$ under $B_{\gD}$, reduction by $N$ at zero imposes the parabolic residue condition and gives $T^*\cV$. Its remaining moment map $\mu_M$ records the Levi residue. Note that $M$ acts trivially on $\zM$, so the central parameter survives the subsequent quotient by $M$. 

\Needspace{12\baselineskip}
\begin{theoremB}[Proposition~\ref{prop:smoothfamily} and Theorem~\ref{thm:simultaneous}]
The projection
\[
 q:\widetilde\cL=\mu_M^{-1}(\zM+\cO)/M\longrightarrow\zM
\]
is smooth and surjective, with symplectic fibers of dimension $2\dim\cB+\dim\gD-\dim\lalg+\dim\cO$. For $m\in\zM^{\reg}=\{m:Z_{\GD}(m)=M\}$, its fiber is symplectomorphic to $\cL_{\GD\cdot(m+e)}$, where $e\in\cO$. The open part of the zero fiber with residue in $\mathscr O_0$ is a finite \'etale symplectic cover of $\cL_{\mathscr O_0}$, of generalized Springer degree.
\end{theoremB}

Here $Z_{\GD}(m)$ denotes the centralizer of $m$. For fixed $m\in\zM^{\reg}$, the residue determines the compatible parabolic reduction uniquely, whereas an induced nilpotent residue can admit several; the generalized Springer degree counts them over a general point of $\mathscr O_0$. The closure of $\cO$ includes limits of the nilpotent Levi residue. Local cotangent product charts give flatness and reduced fibers, and the projectivity of $\GD/P_M$ makes forgetting the reduction proper (Propositions~\ref{prop:orbitclosurefamily} and~\ref{prop:propercomparison}). 

Comparison of the symplectic classes requires a differentiable identification of the fibers. A connection on $\cV\to\cBP$ separates cotangent vectors into base and fiber directions; subtracting $m$ from the Levi residue then gives the common manifold
\[
 Y=T^*\cBP\times_{\cBP}(\cV\times_M\cO).
\]
The reduced forms on $Y$ are denoted by $\Omega_m$. Sternberg's coupling formula \cite{Sternberg1977} expresses the $m$-dependent term by pairing $m$ with the connection curvature; hence the classes vary affinely, as in Duistermaat--Heckman \cite{DH}. Nilpotent scaling makes $\Omega_0$ exact, so the constant term vanishes. Such periods also describe formal symplectic deformations under the cohomological hypotheses of Kaledin--Verbitsky \cite{KaledinVerbitsky2002}.

\begin{theoremC}[Theorem~\ref{thm:periodvariation}]
There is a connection-independent class $c\in H^2(Y,\C)\otimes\zM^*$, determined by the characteristic classes of the associated principal $M/[M,M]$-bundle, such that $[\Omega_m]=\langle m,c\rangle$.
\end{theoremC}

Therefore, we conclude that, together, Theorem~C and the equality of orbit dimensions verify the period and dimension conditions of Balasubramanian--Distler. Theorem~A also turns fixed orbit dimension into fixed Poisson rank. Consequently, \emph{sheets}, the irreducible components of loci of fixed orbit dimension, organize the corresponding rank loci in $\cP$. The sheet theory of Borho--Kraft \cite{BK}, Borho \cite{Borho}, and Broer \cite{BroerLectures} applies through the local cotangent product charts, which carry these irreducible components to one another (see Theorem~\ref{thm:sheetstrata}).

Finally, the parabolic reduction remains part of the data after fixing the Hitchin value. In Yun's global Springer theory, such additional structure supports Springer-type actions on fiber cohomology \cite{Yun2011}. For our family, the Hitchin map records the conjugacy invariants of the central residue. Retaining $m$ itself therefore requires adjoining it to the Hitchin target; Appendix~B constructs this target and establishes compatibility with the regular- and zero-fiber comparisons.

The paper is organized as follows. Section~\ref{sec:leviresidues} reviews nilpotent induction and studies residue orbits with a compatible parabolic reduction. Section~\ref{sec:sheets} relates these residue families to sheets and formulates the dimension and period conditions. In Section~\ref{sec:level-leaves} we describe level structures, the coresidue moment map, and the symplectic leaves of the logahoric quotient. In Section~\ref{sec:simultaneous} we construct the simultaneous parabolic reduction and compare its fibers with the residue-orbit reductions, while in Section~\ref{sec:periods} we compute the variation of the symplectic class. In the final Section~\ref{sec:examples}, we illustrate the construction with a few explicit examples in types~$A$, $D$, and $E$.

\section{Levi Residues and Induced Nilpotent Orbits}\label{sec:leviresidues}

\subsection{The Levi decomposition and the central parameter}

Consider a group $\GD$ which is connected, complex, and reductive, with Lie algebra $\gD$ and a fixed nondegenerate invariant symmetric form $B_{\gD}$. A closed subgroup $P_M\subset\GD$ is \emph{parabolic} if $\GD/P_M$ is projective. The Levi decomposition is $P_M=MN$, where $N$ is its largest connected normal unipotent subgroup and $M$ is a reductive subgroup mapping isomorphically to $P_M/N$. The subgroup $M$ is a \emph{Levi factor}. The corresponding Lie algebras are $\lalg=\Lie(M)$, $\palg=\Lie(P_M)$, and $\nalg=\Lie(N)$; the last is the nilradical of $\palg$. The Levi projection $\operatorname{pr}_{\lalg}:\palg\to\lalg$ gives the vector-space decompositions
\[
 \palg=\lalg\oplus\nalg,
 \qquad
 \lalg=\zM\oplus[\lalg,\lalg],
 \qquad
 \zM=\zalg(\lalg).
\]
Here $\zM$ is the center, $[\lalg,\lalg]$ the semisimple derived algebra, and $\nalg$ an ideal in $\palg$. We write $\Ad(g)$ for the differential of conjugation by $g$, $\ad(a)b=[a,b]$, and $\langle\Ad_g^*\xi,a\rangle=\langle\xi,\Ad_{g^{-1}}a\rangle$ for the coadjoint action.

In a faithful representation, Jordan decomposition gives the unique expression $x=x_s+x_n$ with $x_s$ diagonalizable, $x_n$ nilpotent, and $[x_s,x_n]=0$. These are its semisimple and nilpotent parts. The Lie algebra centralizer is $\mathfrak z_{\gD}(x)=\{a\in\gD:[a,x]=0\}$.

We fix a nilpotent $M$-orbit $\cO\subset[\lalg,\lalg]$. For $m\in\zM$, the $P_M$-invariant residue condition $\operatorname{pr}_{\lalg}(y)\in m+\cO$ is
\[
 y=m+e+n_y,\qquad e\in\cO,\quad n_y\in\nalg.
\]

The centralizer $Z_{\GD}(m)=\{g:\Ad(g)m=m\}$ defines the open subset
\[
 \zM^{\reg}:=\{m\in\zM:Z_{\GD}(m)=M\}.
\]
Relative to a maximal torus of $M$, this is the complement of the hyperplanes $\alpha(m)=0$ for roots outside $M$; the roots of $M$ vanish identically on $\zM$. Thus $0\notin\zM^{\reg}$ when $M\ne\GD$. The normalizer is $N_{\GD}(M)=\{g:gMg^{-1}=M\}$. Conjugate central parameters are identified by the finite group
\[
 N_{\GD}(M,\cO)
 =\{g\in N_{\GD}(M):\Ad(g)\cO=\cO\},
 \qquad Q_M=N_{\GD}(M,\cO)/M,
\]
acting on $\zM^{\reg}$. This group is finite because the identity component of $N_{\GD}(M)$ is $M$.

\begin{definition}[Jordan class]
The \emph{Jordan class} associated with $(\lalg,\cO)$ is the locally closed subset
\[
 J_{\GD}(\lalg,\cO):=\GD\cdot(\zM^{\reg}+\cO)\subset\gD.
\]
Its elements have commuting semisimple and nilpotent parts conjugate to $m\in\zM^{\reg}$ and $e\in\cO$, respectively. Thus a Jordan class fixes the semisimple centralizer and the nilpotent orbit inside it.
\end{definition}

\begin{proposition}[Centralizer and orbit dimension]\label{prop:centralizer}
For $m\in\zM^{\reg}$ and $e\in\cO$,
\[
 Z_{\GD}(m+e)=Z_M(e),
 \qquad
 \mathfrak z_{\gD}(m+e)=\mathfrak z_{\lalg}(e).
\]
Consequently,
\begin{equation}\label{eq:sheetorbitdimension}
 \dim\GD\cdot(m+e)
 =\dim\gD-\dim\lalg+\dim\cO.
\end{equation}
\end{proposition}

\begin{proof}
Uniqueness of Jordan decomposition gives
\[
 Z_{\GD}(m+e)=Z_{\GD}(m)\cap Z_{\GD}(e)=Z_M(e).
\]
Differentiating gives the Lie algebra identity. The dimension formula then follows from $\dim\cO=\dim\lalg-\dim\mathfrak z_{\lalg}(e)$.
\end{proof}

\begin{example}[A two-block residue in $\mathrm{SL}_4$]\label{ex:SL4running}
In this example, $\GD=\mathrm{PGL}_4$, represented by trace-zero matrices through its $\mathrm{SL}_4$ cover. The Levi $M$ is the image of $S(\mathrm{GL}_2\times\mathrm{GL}_2)$, and $\cO=0$. Then $\dim\gD=15$, $\dim\lalg=7$, and
\[
 m(a)=\begin{pmatrix}aI_2&0\\0&-aI_2\end{pmatrix},\qquad a\in\C,
\]
parametrizes $\zM$. For $a\ne0$, its two eigenspaces determine the centralizer $M$, giving $\zM^{\reg}=\C^*$ and orbit dimension $8$. Interchanging the blocks gives $Q_M\cong\mathfrak S_2$, acting by $a\mapsto-a$. The Jordan-class orbits are therefore parametrized by $\C^*/\{a\sim-a\}$; the first two-plane selects the $a$-eigenspace.
\end{example}

\subsection{Nilpotent induction and the zero parameter}

At $m=0$, the residue condition becomes $y=e+n_y$.

\begin{definition}[Nilpotent induction]\label{def:nilpotentinduction}
The \emph{induced orbit} $\Ind_{\lalg}^{\gD}(\cO)$ is the unique nilpotent $\GD$-orbit meeting $\cO+\nalg$ in an open dense subset. The Lusztig--Spaltenstein induction theorem establishes existence, uniqueness, and independence of the parabolic with Levi factor $M$ \cite{LS,CM}.
\end{definition}

For $\cO=0$, this is the \emph{Richardson orbit}: the $\GD$-orbit through the open dense $P_M$-orbit in $\nalg$.

\begin{proposition}[Dimension of the induced orbit]\label{prop:induceddimension}
For every nilpotent $M$-orbit $\cO\subset[\lalg,\lalg]$,
\begin{equation}\label{eq:induceddimension}
 \dim\Ind_{\lalg}^{\gD}(\cO)
 =\dim\gD-\dim\lalg+\dim\cO.
\end{equation}
This is also the dimension of each orbit in $J_{\GD}(\lalg,\cO)$.
\end{proposition}

\begin{proof}
Induction gives $\dim\Ind_{\lalg}^{\gD}(\cO)=\dim\cO+2\dim\nalg$ \cite{LS}. Since the decomposition into $\lalg$ and the two opposite nilradicals gives $2\dim\nalg=\dim\gD-\dim\lalg$, this proves the dimension formula. Combining it with Proposition~\ref{prop:centralizer} gives the last assertion.
\end{proof}

\begin{notation}[Partitions and nilpotent orbits]
A partition $\lambda=[\lambda_1,\lambda_2,\ldots]$ of $N$ is a non-increasing sequence of positive integers summing to $N$; $[a^r,b^s]$ records multiplicities. Its transpose satisfies $\lambda_j^{\mathsf t}=\#\{i:\lambda_i\ge j\}$. In $\mathfrak{sl}_N$, $\mathcal O_\lambda$ has Jordan blocks of sizes $\lambda_i$ and
\begin{equation}\label{eq:typeAorbitdimension}
 \dim\mathcal O_\lambda=N^2-\sum_j(\lambda_j^{\mathsf t})^2.
\end{equation}
In $\mathfrak{so}_N$, every even part must have even multiplicity; a partition with all parts even gives two $\mathrm{SO}_N$-orbits of the same dimension. With $m_i(\lambda)$ the multiplicity of $i$,
\begin{equation}\label{eq:typeDorbitdimension}
 \dim\mathcal O_\lambda
 =\frac{N(N-1)}2-\frac12\left(
 \sum_j(\lambda_j^{\mathsf t})^2-\sum_{i\ \mathrm{odd}}m_i(\lambda)\right).
\end{equation}
For partitions of the same integer, dominance $\mu\leq\lambda$ means $\sum_{i=1}^k\mu_i\leq\sum_{i=1}^k\lambda_i$ for every $k$, with missing parts zero. The $D$-collapse of a partition of an even integer is the largest partition below it in dominance order having even multiplicity for every even part. For exceptional orbits, we use the Bala--Carter notation \cite[Chapters~5--8]{CM}. In $\mathfrak e_6$, $A_2$ denotes the orbit of a nilpotent element with one size-three Jordan block in an $\mathfrak{sl}_3$ Levi factor. The orbit $E_8(b_6)$ is specified by the induction formulas in Section~\ref{subsec:E8families}.
\end{notation}

In the two-block example,
\[
 y(B)=\begin{pmatrix}0&B\\0&0\end{pmatrix},
 \qquad \text{for }B\in\operatorname{Mat}_{2\times2}(\C),
\]
parametrizes $\nalg$. On the dense open subset $\det B\ne0$, we have $y(B)^2=0$ and $\operatorname{rank}y(B)=2$, thus giving
\[
 \Ind_{\lalg}^{\mathfrak{sl}_4}(0)=\mathcal O_{[2^2]},
 \qquad \dim\mathcal O_{[2^2]}=16-(2^2+2^2)=8.
\]

\subsection{Residues with a parabolic reduction}

For a left $P_M$-space $Y$, the associated bundle $\GD\times^{P_M}Y\to\GD/P_M$ is the quotient by
\[
 (h,y)\cdot p=(hp,p^{-1}\cdot y).
\]
For adjoint-stable $Y\subset\gD$, $[h;y]$ records the parabolic $hP_Mh^{-1}$ and residue $\Ad(h)y$. The same quotient convention defines $W\times_HY$ for a right $H$-torsor $W$ and a left $H$-space $Y$.

We retain both the residue and a parabolic reduction compatible with it. Together they form the family
\[
 \widetilde\cJ
 :=\GD\times^{P_M}
 \{(m,y)\in\zM\times\palg:
       \operatorname{pr}_{\lalg}(y)\in m+\cO\},
\]
where $P_M$ fixes $m$ and conjugates $y$. Its fiber and the morphism forgetting the parabolic reduction are
\[
 \widetilde\cJ_m
 =\GD\times^{P_M}(m+\cO+\nalg),
\]
\begin{equation}\label{eq:forgetparabolic}
 \nu:\widetilde\cJ\longrightarrow\zM\times\gD,
 \qquad [h;(m,y)]\longmapsto(m,\Ad(h)y).
\end{equation}
The second component on the fiber at $m$ is denoted by $\nu_m$; $B_{\gD}$ identifies its residue values with elements of $\gD^*$.

\begin{lemma}[Polynomial inversion along a filtration]\label{lem:filteredinverse}
A finite-dimensional vector space has a filtration $V=V_1\supset V_2\supset\cdots\supset V_{r+1}=0$ and a polynomial map $F:V\to V$ with $F(0)=0$. We assume that each $V_i/V_{i+1}$ has a linear automorphism $A_i$ satisfying, for $x\in V$ and $v\in V_i$,
\[
 F(x+v)-F(x)\equiv A_i(v)\pmod{V_{i+1}},
\]
where $A_i(v)$ denotes the image of the class of $v$. Then $F$ is a polynomial automorphism.
\end{lemma}

\begin{proof}
Linear splittings give coordinates adapted to the filtration. In the equation $F(x)=y$, the component of $x$ modulo $V_2$ is determined by $A_1^{-1}$. If $x_{<i}$ is the sum of the first $i-1$ chosen components, the error $y-F(x_{<i})$ lies in $V_i$. Its class in $V_i/V_{i+1}$ determines the next correction uniquely through
\[
 A_i(v_i)\equiv y-F(x_{<i})\pmod{V_{i+1}}.
\]
Each correction is polynomial in $y$, being a fixed linear inverse applied to a polynomial in the preceding components. They therefore give the polynomial inverse.
\end{proof}

\begin{lemma}[Removal of the nilradical component]\label{lem:unipotentorbit}
For $m\in\zM^{\reg}$ and $e\in\cO$, the map
\[
 N\longrightarrow m+e+\nalg,
 \qquad n\longmapsto\Ad(n)(m+e),
\]
is an isomorphism. Consequently, $P_M\cdot(m+e)=m+\cO+\nalg$.
\end{lemma}

\begin{proof}
A Borel subgroup of $P_M$, meaning a maximal connected solvable subgroup, and a maximal torus of its intersection with $M$ determine positive roots. The subspace $F^i\nalg$ is the sum of root spaces whose coefficients on the simple roots omitted from $M$ have total at least $i$. Root addition gives
\[
 \nalg=F^1\nalg\supset F^2\nalg\supset\cdots\supset F^{r+1}\nalg=0,
 \qquad [F^i\nalg,F^j\nalg]\subset F^{i+j}\nalg.
\]
The Levi algebra preserves this filtration. Regularity makes all eigenvalues of $\ad(m)$ on $\nalg$ nonzero. The commuting nilpotent operator $\ad(e)$ preserves its eigenspaces, so on the $c$-eigenspace, $\ad(m+e)=c\,\mathrm{id}+\ad(e)$ has inverse given by the finite geometric series in $-c^{-1}\ad(e)$. It is therefore invertible on each filtration quotient.

The algebraic exponential isomorphism $\nalg\to N$ turns the orbit map into
\[
 \Psi(X)=\Ad(\exp X)(m+e)-(m+e)
       =[X,m+e]+\frac12[X,[X,m+e]]+\cdots.
\]
This finite series satisfies, for $V\in F^i\nalg$,
\[
 \Psi(X+V)-\Psi(X)\equiv[V,m+e]\pmod{F^{i+1}\nalg},
\]
because every further term containing $V$ also contains an element of $F^1\nalg$. Indeed, its graded maps are the invertible operators $-\ad(m+e)$. Thus Lemma~\ref{lem:filteredinverse} proves the first assertion; acting by $M$ then gives the second.
\end{proof}

\begin{example}[Two-block residues with a chosen two-plane]\label{ex:SL4parabolic}
For the parabolic stabilizing the first two-plane in Example~\ref{ex:SL4running},
\[
 n(X)=\begin{pmatrix}I_2&X\\0&I_2\end{pmatrix}
\]
{gives}
\begin{equation}\label{eq:SL4unipotentcalculation}
 \Ad(n(X))m(a)=\begin{pmatrix}aI_2&-2aX\\0&-aI_2\end{pmatrix}.
\end{equation}
For $a\ne0$, an upper-right block $B$ uniquely determines $X=-B/(2a)$.

The fiber $\widetilde\cJ_a$ consists of pairs $(V,T)$ with $V\subset\C^4$ a two-plane, $T|_V=a\,\mathrm{id}$, and induced operator $-a\,\mathrm{id}$ on $\C^4/V$. The map $\nu_a$ forgets $V$. For $a\ne0$, the block form gives $T^2=a^2I$, so $T$ is semisimple and $V$ its uniquely determined $a$-eigenspace. Thus $\nu_a$ identifies the fiber with the orbit of $m(a)$.

At $a=0$, the conditions are
\[
 \operatorname{im}T\subset V\subset\ker T.
\]
For fixed $V$, these operators form $\operatorname{Hom}(\C^4/V,V)=T_V^*\operatorname{Gr}(2,4)$, where $\operatorname{Gr}(2,4)$ is the Grassmannian of two-planes in $\C^4$. Hence the zero fiber is $T^*\operatorname{Gr}(2,4)$, with image $\overline{\mathcal O_{[2^2]}}$, the square-zero matrices of rank at most two. In particular, rank two forces $V=\operatorname{im}T=\ker T$, giving degree one over $\mathcal O_{[2^2]}$; the fiber over $T=0$ is the whole Grassmannian.
\end{example}

\subsection{The finite cover over the induced orbit}

The embedding $[h;y]\mapsto(hP_M,\Ad(h)y)$ identifies $\GD\times^{P_M}(\overline\cO+\nalg)$ with a closed subvariety of $\GD/P_M\times\gD$. Consequently, projectivity of $\GD/P_M$ and the dimension formula give a proper, generically finite morphism
\[
 \overline\nu_0:
 \GD\times^{P_M}(\overline\cO+\nalg)
 \longrightarrow\overline{\Ind_{\lalg}^{\gD}(\cO)}.
\]
Its generic degree counts the compatible parabolic reductions of a general residue in the induced orbit. This is the \emph{generalized Springer degree} of $(P_M,\cO)$. A finite \'etale cover has finite fibers and no ramification; over $\C$ it is locally a holomorphic isomorphism. For $\cO=0$ this is the generalized Springer map from $T^*(\GD/P_M)$; for a Borel its degree is one.

\begin{lemma}[The cover over the induced orbit]\label{lem:inducedorbitcover}
The induced orbit and the spaces above it are
\[
 \mathscr O_0=\Ind_{\lalg}^{\gD}(\cO),\qquad
 S_0=(\cO+\nalg)\cap\mathscr O_0,
 \qquad
 \widetilde{\mathscr O}_0=\GD\times^{P_M}S_0.
\]
Forgetting the parabolic reduction gives a finite \'{e}tale surjection
\[
 \nu_0:\widetilde{\mathscr O}_0\longrightarrow\mathscr O_0.
\]
Its source is a smooth irreducible homogeneous $\GD$-space, and its degree is the generalized Springer degree of $(P_M,\cO)$.
\end{lemma}

\begin{proof}
Each of the finitely many $M$-orbits $\cO'\subset\overline\cO\setminus\cO$ gives a stratum of dimension
\[
 \dim\bigl(\GD\times^{P_M}(\cO'+\nalg)\bigr)
 =2\dim\nalg+\dim\cO'
 <2\dim\nalg+\dim\cO=\dim\mathscr O_0.
\]
Its $\GD$-stable image cannot meet $\mathscr O_0$, since it would then contain that orbit. Thus
\[
 \overline\nu_0^{-1}(\mathscr O_0)=\widetilde{\mathscr O}_0,
\]
so $\nu_0$ is proper. Since the map is dominant and the dimensions are equal, it is generically finite; equivariance then makes every fiber zero-dimensional. It is therefore finite by properness.

Induction makes $S_0$ open dense in $\cO+\nalg$, so $\widetilde{\mathscr O}_0$ is smooth irreducible of dimension $\dim\mathscr O_0$. Each $\GD$-orbit surjects onto $\mathscr O_0$ with finite fibers, hence has full dimension and is open. Irreducibility forces a single orbit.

For $\widetilde y$ above $y$, the stabilizers satisfy $H'=(\GD)_{\widetilde y}\subset H=(\GD)_y$. The map is
\[
 \GD/H'\longrightarrow\GD/H.
\]
Finiteness gives $[H:H']<\infty$, hence $\Lie(H')=\Lie(H)$. Its differential between smooth homogeneous spaces is therefore an isomorphism, proving \'{e}taleness. The degree equals that of $\overline\nu_0$.
\end{proof}

We observe that transitivity on $\GD\times^{P_M}S_0$ implies that $S_0$ is a single $P_M$-orbit: an element taking $[1;y_1]$ to $[1;y_2]$ must lie in $P_M$. Since $H'$ has finite index in $H$, it contains $H^\circ$. The degree is $[H:H']=[H/H^\circ:H'/H^\circ]$, where $H^\circ$ is the identity component and $H/H^\circ$ the component group. These cosets parametrize the parabolic reductions over $y$.

\begin{theorem}[Regular fibers and the induced nilpotent orbit]\label{thm:parabolicfibers}
For $m\in\zM^{\reg}$ and $e\in\cO$, the morphism $\nu_m$ is an isomorphism
\[
 \widetilde\cJ_m\xrightarrow{\sim}\GD\cdot(m+e).
\]
At $m=0$, its image is $\GD\cdot(\cO+\nalg)$, containing $\Ind_{\lalg}^{\gD}(\cO)$ as a dense open orbit. Every fiber of $\widetilde\cJ\to\zM$ has dimension $\dim\gD-\dim\lalg+\dim\cO$.
\end{theorem}

\begin{proof}
For regular $m$, Lemma~\ref{lem:unipotentorbit} and Proposition~\ref{prop:centralizer} give
\[
 \GD\times^{P_M}(m+\cO+\nalg)
 \cong\GD\times^{P_M}(P_M/Z_M(e))
 \cong\GD/Z_M(e)
 \cong\GD\cdot(m+e),
\]
Under these identifications, $\nu_m$ is the orbit map. At zero, induction identifies the dense orbit in $\GD\cdot(\cO+\nalg)$; it is open because it is open in its closure. Finally, for every $m$, the associated-bundle dimension is
\[
 \dim(\GD/P_M)+\dim\cO+\dim\nalg
 =2\dim\nalg+\dim\cO
 =\dim\gD-\dim\lalg+\dim\cO.
\]
\end{proof}

\subsection{The parameter space of Jordan-class orbits}

A \emph{geometric quotient} is a surjective invariant morphism with the group orbits as geometric fibers, the quotient topology, and regular functions equal to the invariant functions on inverse images of open subsets.

\begin{proposition}[The geometric quotient of a Jordan class]\label{prop:jordanclassquotient}
For a fixed $e\in\cO$, the Jordan class $J_{\GD}(\lalg,\cO)$ carries its reduced locally closed structure. There are canonical isomorphisms
\begin{equation}\label{eq:jordanassociatedbundle}
 J_{\GD}(\lalg,\cO)
 \simeq\bigl(\GD\times^M(\zM^{\reg}\times\cO)\bigr)/Q_M
 \simeq\GD\times^{N_{\GD}(M,\cO)}(\zM^{\reg}\times\cO).
\end{equation}
The $Q_M$-action on $\zM^{\reg}$ is free with finite \'{e}tale quotient. The morphism
\begin{equation}\label{eq:jordanorbitquotient}
 \pi_J:J_{\GD}(\lalg,\cO)\longrightarrow\zM^{\reg}/Q_M,
 \qquad \Ad(g)(m+e')\longmapsto[m],
\end{equation}
is the smooth geometric quotient by $\GD$, with reduced geometric fiber $\GD\cdot(m+e)$ over $[m]$.
\end{proposition}

\begin{proof}
On $E_J=\GD\times^M(\zM^{\reg}\times\cO)$, a representative $n\in N_{\GD}(M,\cO)$ acts by
\[
 (nM)\cdot[g;(m,e')]_M
 =[gn^{-1};(\Ad(n)m,\Ad(n)e')]_M.
\]
Replacing $n$ modulo $M$ leaves the point unchanged. The action is free, since a fixed point forces $gn^{-1}M=gM$, hence $n\in M$. Successive quotients by $M$ and $Q_M$ give the second associated bundle in \eqref{eq:jordanassociatedbundle}.

The orbit map induces a surjective morphism
\[
 \Phi:E_J/Q_M\longrightarrow J_{\GD}(\lalg,\cO),
 \qquad[g;(m,e')]\longmapsto\Ad(g)(m+e').
\]
If $\Ad(g_1)(m_1+e_1)=\Ad(g_2)(m_2+e_2)$, uniqueness of Jordan decomposition matches the semisimple and nilpotent parts. Hence $g_2^{-1}g_1$ normalizes their centralizer $M$ and preserves $\cO$. The representatives therefore define the same point of $E_J/Q_M$, proving injectivity over every algebraically closed extension.

The differential can be computed in the decomposition $\gD=\lalg\oplus\mathfrak u$, where $\mathfrak u$ is the sum of the nonzero $\ad(m)$-eigenspaces. Changing the group coordinate, the central parameter, and the nilpotent residue gives the differential at $m+e'$:
\[
 \mathfrak u\oplus\zM\oplus[\lalg,e']\longrightarrow\gD,
 \qquad(A,v,[B,e'])\longmapsto[A,m+e']+v+[B,e'].
\]
Invertibility of $\ad(m+e')$ on $\mathfrak u$, as in Lemma~\ref{lem:unipotentorbit}, forces $A=0$ in its kernel. Since the remaining terms lie in the center and derived algebra of $\lalg$, both vanish. The image of this differential lies in the tangent space to the Jordan class. The source $E_J/Q_M$ and target $J_{\GD}(\lalg,\cO)$ are smooth of dimension $\dim\GD-\dim Z_M(e')+\dim\zM$; smoothness of Jordan classes is proved in \cite{BroerLectures}. Thus $\Phi$ is \'{e}tale. Injectivity on geometric points makes it an open immersion, and surjectivity makes it an isomorphism.

A representative of $q\in Q_M$ fixing $m\in\zM^{\reg}$ belongs to $Z_{\GD}(m)=M$, so $q=1$. Its affine quotient is therefore finite \'{e}tale. After base change by $\zM^{\reg}\to\zM^{\reg}/Q_M$, $\pi_J$ becomes
\[
 \zM^{\reg}\times\GD/Z_M(e)\longrightarrow\zM^{\reg},
\]
using $\cO\cong M/Z_M(e)$. This smooth projection has the ambient orbits as fibers; smoothness and the geometric-quotient property descend along the finite \'{e}tale cover.
\end{proof}
\section{Sheets and Central Residue Families}\label{sec:sheets}

\subsection{Orbit dimensions and sheet data}

A \emph{sheet} is an irreducible component of an orbit-dimension locus
\[
 (\gD)_{(d)}=\{x\in\gD:\dim\GD\cdot x=d\},\qquad d\geq0.
\]
The invariant form identifies this locus with $(\gD^*)_{(d)}$. For the Jordan class $J_{\GD}(\lalg,\cO)$, Proposition~\ref{prop:centralizer} gives the common orbit dimension $\dim\gD-\dim\lalg+\dim\cO$.

\begin{definition}\label{def:sheetdatum}
A nilpotent orbit $\cO$ is called \emph{rigid in $\lalg$} if it is not induced from a proper Levi subalgebra. A \emph{sheet datum} is a $\GD$-conjugacy class of pairs $(\lalg,\cO)$ with $\cO$ rigid in $\lalg$.
\end{definition}

The following result provides a classification of the sheet data and is due to Borho: 
\begin{theorem}[Borho \cite{Borho}]\label{thm:borho}
Each sheet datum $(\lalg,\cO)$ determines a unique sheet with open dense Jordan class $J_{\GD}(\lalg,\cO)$; every sheet arises uniquely this way.
\end{theorem}

The \emph{rank} of this sheet is $\dim\zM$, the number of central parameters in its dense Jordan class.
The other Jordan classes in a sheet have larger semisimple centralizers and nilpotent parts determined by induction \cite{Borho,Broer,Summerfield,CM,dGE}. For a rigid $\cO$, the subset $\zM+\cO$ is itself a sheet of $\lalg$ \cite{Borho}.
The zero orbit is rigid, since induction from a proper Levi adds $2\dim\nalg$ to the orbit dimension. A sheet with datum $(\lalg,0)$ is called a \emph{Dixmier sheet}; its dense Jordan class is semisimple, and its nilpotent orbit is the Richardson orbit $\Ind_{\lalg}^{\gD}(0)$ \cite{BK,LS}.

\begin{proposition}[The induced nilpotent orbit in a sheet]\label{prop:inducedboundary}
The induced orbit $\Ind_{\lalg}^{\gD}(\cO)$ has the same dimension as the orbits in $J_{\GD}(\lalg,\cO)$. If $\cO$ is rigid and $S$ is the sheet containing this Jordan class, then $\Ind_{\lalg}^{\gD}(\cO)\subset S$.
\end{proposition}

\begin{proof}
The dimension equality is provided by Proposition~\ref{prop:induceddimension}. For any $y=e+n_y\in(\cO+\nalg)\cap\Ind_{\lalg}^{\gD}(\cO)$ and $m\in\zM^{\reg}$, Lemma~\ref{lem:unipotentorbit} now gives
\[
 tm+e+n_y\in P_M\cdot(tm+e)\subset J_{\GD}(\lalg,\cO)
 \quad(t\ne0).
\]
Taking $t=0$ and conjugating places the induced orbit in $\overline{J_{\GD}(\lalg,\cO)}$. Since both have orbit dimension $d=\dim\gD-\dim\lalg+\dim\cO$, the induced orbit lies in the closure of the Jordan class inside $(\gD)_{(d)}$, which is $S$.
\end{proof}

\subsection{Dimension and period conditions}

A holomorphic symplectic form on a complex manifold $Y$ is a closed, nondegenerate holomorphic two-form. Its periods are its integrals over closed two-cycles; they determine its class in complex de Rham cohomology $H^2(Y,\C)$.

\begin{definition}[Linear periods and mass-like deformations]\label{def:linearperiods}
A smooth family of holomorphic symplectic manifolds over a complex vector space $V$ has \emph{linear periods} if, under a differentiable trivialization with fiber $Y$, the map $v\mapsto[\omega_v]-[\omega_0]\in H^2(Y,\C)$ is complex-linear. A residue family parametrized by a Levi center, with symplectic fibers, is called \emph{mass-like} if its regular residue orbits and its induced nilpotent orbit have equal dimension and its symplectic forms have linear periods \cite[Eqns.~(2.13)--(2.14)]{BD}.
\end{definition}

\subsection{Orbit closures at regular parameters}
\label{subsec:centralparameters}

\begin{proposition}[Orbit closures at regular parameters]\label{prop:residualsingularities}
For $m\in\zM^{\reg}$ and $e\in\cO$, there is an isomorphism of reduced varieties
\[
 \overline{\GD\cdot(m+e)}
 =\GD\cdot(m+\overline\cO)
 \simeq\GD\times^M\overline\cO.
\]
Over $\GD\cdot m\simeq\GD/M$, this variety is \'etale locally a product with $\overline\cO$. It is smooth when $\cO=0$ and singular along $\GD\cdot m$ when $\cO\ne0$.
\end{proposition}

\begin{proof}
The morphism
\[
 \GD\times^{P_M}(m+\overline\cO+\nalg)\longrightarrow\gD
\]
is proper: its source is closed in $\GD/P_M\times\gD$. The inverse in Lemma~\ref{lem:unipotentorbit} works uniformly for $e\in\overline\cO$. On a nonzero $\ad(m)$-eigenspace, $(c\,\mathrm{id}+\ad(e))^{-1}$ is a finite geometric series with a uniform nilpotence bound. The polynomial inversion along the filtration therefore gives
\[
 N\times\overline\cO\xrightarrow{\sim}m+\overline\cO+\nalg,
 \qquad(n,e)\longmapsto\Ad(n)(m+e).
\]
The associated bundle is consequently $\GD\times^M(m+\overline\cO)$, with closed image $C_m=\GD\cdot(m+\overline\cO)$.

In a faithful representation of dimension $r$, the Chinese remainder theorem gives a polynomial $P(t)$ with $P(t)\equiv\lambda\pmod{(t-\lambda)^r}$ for each eigenvalue $\lambda$ of $m$. For $y\in C_m$, the nilpotent part on each generalized eigenspace has order at most $r$, so $P(y)=y_s$. Thus $y\mapsto y_s$ is a regular map to $\GD\cdot m\cong\GD/M$. Composing with $\GD/M\to\GD/P_M$ recovers the compatible parabolic reduction and gives the inverse to the proper morphism.

The orbit $\GD\cdot(m+e)$ is dense in $C_m$. Moreover, the principal $M$-bundle $\GD\to\GD/M$ gives the local product description with $\overline\cO$.

A nonzero nilpotent orbit closure is singular at zero. Indeed, an irreducible affine cone smooth at its vertex equals its tangent space there. This would make $\overline\cO$ a nonzero $M$-invariant vector subspace of $[\lalg,\lalg]$, hence an ideal consisting entirely of nilpotent elements. A semisimple Lie algebra has no such nonzero ideal.
\end{proof}

\subsection{Levi data inside a sheet}\label{subsec:restrictedfamilies}

If $\cO$ is nonrigid in $\lalg$, successive induction gives a proper Levi $\lalg_s\subset\lalg$ and a rigid orbit $\cO_s$ such that
\[
 \cO=\Ind_{\lalg_s}^{\lalg}(\cO_s),\qquad
 \zalg(\lalg)\subsetneq\zalg(\lalg_s).
\]
The Jordan-class closure relation gives
\[
 J_{\GD}(\lalg,\cO)
 \subset\overline{J_{\GD}(\lalg_s,\cO_s)}
\]
\cite{Borho,LS}. Moreover, the induction dimension formula gives equal orbit dimensions, so the first class lies in the sheet with dense class $J_{\GD}(\lalg_s,\cO_s)$. Its central parameter space has smaller dimension than the rank of that sheet.

\begin{example}[Two Levi data for $A_2$ in $\mathfrak e_6$]\label{ex:E6central}
The pair $(\mathfrak e_6,A_2)$ has zero-dimensional center, so its family consists of a single fiber. The same orbit $A_2$ is induced from zero in a Levi $\lalg_s$ of derived type $A_5$, with $\dim\lalg_s=35+1=36$. Thus $\dim A_2=78-36=42$, and it lies in the rank-one Dixmier sheet with datum $(\lalg_s,0)$ \cite[Section~6.3.1, Figure~6]{BD}. This second datum gives a one-parameter family with the same induced nilpotent orbit.
\end{example}

\section{Level Structures, Residues, and Symplectic Leaves}
\label{sec:level-leaves}

We start in this section by recalling the definition of a level structure on parahoric torsors from \cite{KZ} and the induced moment map construction on the moduli space of logahoric Higgs torsors. We then move on to study the symplectic leaves of this moduli space and its sheet loci.

Let $X$ be a smooth complex projective curve, let $D=\{x\}$, for a point $x \in X$; moreover, let $G$ be a connected complex semisimple group, and put $\g:=\Lie(G)$. Fix a maximal torus $T\subset G$ and a rational cocharacter $\theta\in Y(T)\otimes_{\mathbb Z}\mathbb Q$ of $T$, where $Y(T):=\operatorname{Hom}(\mathbb G_m,T)$. The cocharacter  determines a parahoric group scheme $\mathcal G_\theta$ (see \cite[Section 2.1]{KSZ}). On the formal disc, write $G_\theta$ for its parahoric subgroup, $G_\theta^+$ for the pro-unipotent radical and put respectively $\mathfrak g_\theta:=\Lie(G_\theta)$ and $\mathfrak g_\theta^+:=\Lie(G_\theta^+)$. Their reductive quotient and Lie algebra are respectively
\[
 L_\theta=G_\theta/G_\theta^+,
 \qquad \widehat{\mathfrak l}_\theta=\Lie(L_\theta).
\]

The effective level group and its Lie algebra are $\GD=L_\theta/Z(G)$ and $\gD=\Lie(L_\theta)$. The quotient $\rho_\theta:\mathfrak g_\theta\to\gD$ has kernel $\mathfrak g_\theta^+=\Lie(G_\theta^+)$, where $\mathfrak g_\theta=\Lie(G_\theta)$; further details appear in Appendix~\ref{sec:parahoric-formulas}. The hyperspecial case considered here however has $\theta=0$ and $G_\theta=G(\C[[z]])$; then $\rho_\theta$ is evaluation at $z=0$. Lastly, a fixed nondegenerate invariant symmetric form $B_\g$ on the Lie algebra $\mathfrak{g}$ restricts to $B_{\gD}$, which identifies $\gD\cong\gD^*$.

\subsection{Torsors and level structures}

A $\mathcal G_\theta$-torsor $\mathcal E$ is an \'{e}tale locally trivial principal bundle for $\mathcal G_\theta$. Its adjoint sheaf $\mathcal E(\g)$ consists of sections associated to the adjoint representation. It has local lattice $\mathfrak g_\theta$ at $x$ and is the usual adjoint bundle elsewhere. A \emph{logarithmic Higgs field} is a section
\[
 \varphi\in H^0\bigl(X,\mathcal E(\g)\otimes K_X(D)\bigr),
\]
where $K_X(D)$ is the line bundle of one-forms with at most a simple pole at $x$; locally $\varphi=A(z)\,dz/z$ with $A(z)\in\mathfrak g_\theta$. The pair $(\mathcal E,\varphi)$ is called a \emph{logahoric $\mathcal G_\theta$-Higgs torsor}.

A \emph{level structure} $\eta$ on a parahoric $\mathcal G_\theta$-torsor is defined as a reduction of the local $G_\theta$-torsor to $G_\theta^+$, equivalently a trivialization of its $L_\theta$-torsor \cite[Definition~3.1]{KZ}. It identifies the residue with an element of $\gD$; forgetting $\eta$ retains its conjugacy class.

For a proper parabolic $Q\subset G$ containing $T$, the subgroup scheme $\mathcal Q_\theta\subset\mathcal G_\theta$ has formal-disc group $Q(\C((z)))\cap G_\theta$ and fiber $Q$ away from $x$. A reduction $\sigma:X\to\mathcal E/\mathcal Q_\theta$ gives a $\mathcal Q_\theta$-torsor $\mathcal E_\sigma$. A character $\chi:Q\to\mathbb G_m$ is antidominant if it has nonpositive pairings with the simple coroots for a Borel in $Q$. The associated line bundle is $L_{\sigma,\chi}=\mathcal E_\sigma\times^{\mathcal Q_\theta}\C_\chi$, with $\C_\chi$ the one-dimensional representation of character $\chi$. Ramanathan stability means
\[
 \deg L_{\sigma,\chi}+\langle\theta,\chi\rangle>0
\]
for every reduction and every nontrivial antidominant character trivial on $Z(G)$. The pairing $\langle\theta,\chi\rangle$ evaluates the character on the rational cocharacter $\theta$; semistability amounts to replacing $>$ by $\geq$ above.

The stability condition for a Higgs pair $(\mathcal{E}, \varphi)$ tests the reductions preserved by $\varphi$, that is, those reductions for which $\varphi$ takes values in $\mathcal E_\sigma(\Lie\mathcal Q_\theta)\otimes K_X(D)$ \cite[Definitions~4.2--4.5]{KSZ}. Thus, stable underlying parahoric $\mathcal G_\theta$-torsors give stable logahoric $\mathcal G_\theta$-Higgs pairs.

We use $\mathcal M_{LH}(X,\mathcal G_\theta)$ for the coarse moduli space of leveled triples $(\mathcal E,\varphi,\eta)$ with stable $\mathcal E$, and $\mathcal M_H(X,\mathcal G_\theta)$ for that of semistable logahoric $\mathcal G_\theta$-Higgs pairs, as in \cite[Sections~5.4--5.5]{KZ}. We also consider the morphism $\ell$ that forgets $\eta$. A $\mathcal G_\theta$-torsor $\mathcal E$ is called \emph{regularly stable} if it is stable and $\operatorname{Aut}(\mathcal E)=Z(G)$.

For simply connected $G$ and genus $g_X\geq3$, the regularly stable locus is nonempty: Parameswaran--Pandey prove that its complement in the torsor stack has codimension at least two \cite[Theorem~7.2.1]{ParameswaranPandey2023}.

The base $\cB$ is a fixed nonempty connected component of the regularly stable torsor moduli space. The space $\cU$ parametrizes its leveled torsors, and the projection
\[
 \pi:\cU\longrightarrow\cB
\]
forgets the level. Right multiplication changes the level; central automorphisms identify choices differing by $Z(G)$, giving the effective right action of $\GD$.

\begin{proposition}\label{prop:levelbundle}\label{rem:whyregularstable}
The $\GD$-action on $\cU$ is free, and $\pi:\cU\to\cB$ is a $\GD$-torsor in the \'{e}tale topology. The spaces $\cB$ and $\cU$ are smooth, and $\dim\cU=\dim\cB+\dim\GD$.
\end{proposition}

\begin{proof}
The levels of $\mathcal E$ form an $L_\theta$-torsor, and their isomorphisms come from $\operatorname{Aut}(\mathcal E)=Z(G)$. Thus $\GD$ acts freely \cite[Theorem~5.14]{KZ}. The regularly stable stack has constant stabilizer $Z(G)$, which is finite \'{e}tale since $G$ is semisimple. It therefore admits universal torsors \'{e}tale locally over $\cB$. Their reductive frames modulo central automorphisms form a $\GD$-torsor. Changes of universal torsor identify these constructions on overlaps; the central ambiguity disappears in $\GD$, so they descend to $\pi$.

Deformations and obstructions are given by the cohomology groups $H^1(X,\mathcal E(\g))$ and $H^2(X,\mathcal E(\g))=0$. Thus, the torsor stack is smooth, with no infinitesimal stabilizers because $\Lie Z(G)=0$. Its finite central stabilizer acts trivially on the adjoint deformation space. The \'{e}tale local coarse quotient is therefore smooth. Since $\cU$ is locally $\cB\times\GD$, it is smooth of dimension $\dim\cB+\dim\GD$.
\end{proof}

\subsection{The deformation sequence and cotangent vectors}

The sheaf $\mathcal E(\g)_\eta$ consists of infinitesimal automorphisms preserving $\eta$. They have zero reductive value at $x$, so its local lattice is $\mathfrak g_\theta^+$. We therefore have
\begin{equation}\label{eq:levelsequence}
 0\longrightarrow\mathcal E(\g)_\eta
 \longrightarrow\mathcal E(\g)
 \longrightarrow i_{x*}\gD\longrightarrow0,
\end{equation}
where $i_x:\{x\}\hookrightarrow X$, and $\mathcal{E}(\mathfrak{g})_{\eta} \subset \mathcal{E}(\mathfrak{g})$ denotes the subsheaf whose sections preserve the parahoric structure at $x$.

By \cite[Proposition~4.1 and Section~5.4]{KZ}, deformation theory identifies
\[
 T_{(\mathcal E,\eta)}\cU
 =H^1(X,\mathcal E(\g)_\eta),
 \qquad
 T_{\mathcal E}\cB=H^1(X,\mathcal E(\g)).
\]
Regular stability gives $H^0=0$ for both sheaves, so \eqref{eq:levelsequence} induces
\[
 0\longrightarrow\gD
 \xrightarrow{\ \delta_{\mathcal E}\ }
 H^1(X,\mathcal E(\g)_\eta)
 \longrightarrow H^1(X,\mathcal E(\g))\longrightarrow0.
\]
Here $\delta_{\mathcal E}$ is the infinitesimal right action, giving the tangent sequence of $\pi$.

Lemma~\ref{lem:parahoricpairing} gives $\mathcal E(\g)_\eta^\vee\otimes K_X\cong\mathcal E(\g)\otimes K_X(D)$. Serre duality therefore identifies the dual of the tangent space with the logarithmic Higgs fields:
\[
 T^*_{(\mathcal E,\eta)}\cU
 =H^0\bigl(X,\mathcal E(\g)\otimes K_X(D)\bigr).
\]
For $\mathcal E(\g)_\eta$, the vanishing of $H^0$ and $H^2$ makes $H^1$ compatible with base change. Relative Serre duality therefore gives the identification in families, yielding the open subspace
\[
 \cX=T^*\cU=\mathcal M_{LH}(X,\mathcal G_\theta)|_{\cB},
\]
where $|_{\cB}$ means that the underlying torsor lies in $\cB$.

On a cotangent bundle $p:T^*W\to W$, the tautological one-form evaluates a tangent vector $v$ by $\lambda_{(w,\alpha)}(v)=\alpha(dp(v))$. The canonical symplectic form is $\omega=-d\lambda$. For a right group action, the cotangent moment map is $\langle\mu(w,\alpha),a\rangle=\alpha(a^\#_w)$, where $a^\#_w=\left.\frac{d}{dt}\right|_{t=0}w\exp(ta)$.

For $W=\cU$, the canonical form is $\omega_{LH}=-d\lambda_{\cU}$. The lifted action is \emph{Hamiltonian}, meaning $\iota_{A^\#}\omega_{LH}=d\langle\mu,A\rangle$.

\subsection{Residue, coresidue, and the moment map}

In a local frame compatible with $\eta$, the Higgs field has the expression $\varphi=A(z)\,dz/z$. Its \emph{reductive residue}, also called its Levi residue, is
\[
 \Res_x(\varphi)=\rho_\theta(A(z))\in\gD.
\]
Changing $\eta$ conjugates this element. Pairing the residue with $B_{\gD}$ converts it to a covector, its \emph{coresidue}:
\[
 \CoRes_x(\varphi)=B_{\gD}\bigl(\Res_x(\varphi),-\bigr)\in\gD^*.
\]

\subsection{Local charts for cotangent quotients}

Proposition~\ref{prop:cotangentquotientcharts} in Appendix~\ref{sec:cotangentquotients} constructs the quotient of $T^*\cU$ and its restrictions to residue conditions. The local formula below describes these quotients on a chart trivializing the level bundle.

A \emph{Poisson bracket} is a skew, $\C$-bilinear bracket on functions satisfying the Jacobi and Leibniz identities. On a smooth space $Z$, it defines a bivector $\Pi$ by $\Pi(df,dg)=\{f,g\}$ and a map $\Pi^\#:T^*Z\to TZ$ by $(\Pi^\#df)(g)=\{f,g\}$. The vector field $\Pi^\#df$ is the Hamiltonian vector field of $f$. A morphism is Poisson if pullback preserves brackets.

A \emph{symplectic leaf} is a maximal connected immersed manifold whose tangent spaces are $\operatorname{im}\Pi^\#$. The bracket induces a nondegenerate symplectic form on each leaf. We have the following result from our previous work \cite{KZ}:

\begin{theorem}[Coresidue moment map]\label{thm:coresidue}\label{rem:momentlocal}
The cotangent lift of the right $\GD$-action on $\cU$ is Hamiltonian, with moment map
\[
 \mu:\cX\longrightarrow\gD^*,\qquad
 \mu(\mathcal E,\varphi,\eta)=\CoRes_x(\varphi).
\]
The forgetful morphism restricts to the quotient
\[
 \ell:\cX\longrightarrow\cP=\cX/\GD
 =\mathcal M_H(X,\mathcal G_\theta)|_{\cB}.
\]
The induced Poisson bracket on $\cP$ is the restriction of the bracket in \cite[Theorem~5.23]{KZ}.
\end{theorem}

\begin{proof}
Under Serre duality, $\langle\mu,A\rangle=\varphi(\delta_{\mathcal E}(A))$. A local lift $\widetilde A$ of $A$ on the disc, together with the zero lift away from $x$, represents $\delta_{\mathcal E}(A)$ by a \v{C}ech cocycle on their overlap. Serre duality pairs this cocycle with $\varphi$ by taking the residue of $B_\g(\varphi,\widetilde A)$, namely $B_{\gD}(\Res_x\varphi,A)$ \cite[Theorem~1.3]{KZ}. The cotangent moment map is therefore the coresidue.

Proposition~\ref{prop:cotangentquotientcharts} constructs $\cP$, whose points are the pairs $(\mathcal E,\varphi)$ obtained by forgetting $\eta$. The symplectic bracket preserves invariant functions and descends to the quotient, where it agrees with \cite[Theorem~5.23]{KZ} on the cotangent locus.
\end{proof}

\begin{lemma}[Smooth surjectivity]\label{lem:musubmersion}
The cotangent projection is $p:T^*\cU\to\cU$. The morphism
\[
 (p,\mu):T^*\cU\longrightarrow\cU\times\gD^*
\]
is smooth and surjective. Over $(u,\xi)$ its fiber is a nonempty affine space of dimension $\dim\cU-\dim\GD$. In particular $\mu$ is smooth and surjective.
\end{lemma}

\begin{proof}
Freeness of the action makes the infinitesimal action a vector-bundle injection $a:\cU\times\gD\to T\cU$. Dualizing therefore gives the surjective bundle map $(p,\mu)$, with affine fibers whose translation space is $\ker(a_u^*)$, of dimension $\dim\cU-\dim\GD$. This map and the projection $\cU\times\gD^*\to\gD^*$ are smooth and surjective, so their composition $\mu$ is too.
\end{proof}

\begin{example}[Local cotangent coordinates]\label{ex:localcotangentmodel}
A level over an \'{e}tale chart $V\to\cB$ gives
\[
 \cP|_V\cong T^*V\times\gD^*.
\]
Its coordinates are the cotangent vector to the torsor moduli and the residue in that level. The transitions are those of Proposition~\ref{prop:cotangentquotientcharts}; an orbit reduction is locally $T^*V\times\mathscr O$.
\end{example}

For the full residue space $\gD^*$, these charts are Poisson isomorphisms with the product of the cotangent and Lie--Poisson structures. Indeed, in the coordinates of Appendix~\ref{sec:cotangentquotients}, with $h\in\GD$ and $A,B\in\gD$, the form computed from $\alpha+\langle\zeta,dh\,h^{-1}\rangle$ gives the cotangent brackets, $\{\langle\zeta,A\rangle,\langle\zeta,B\rangle\}=\langle\zeta,[A,B]\rangle$, and zero mixed brackets.

\subsection{Orbit reductions and all symplectic leaves}

By the vector $A^\#_\xi$ we will denote the derivative of $\Ad_{\exp(tA)}^*\xi$ at $t=0$. The Kirillov--Kostant--Souriau (KKS) form on a coadjoint orbit $\mathscr O\subset\gD^*$ is then defined by
\[
 (\omega_{\mathscr O})_\xi(A^\#_\xi,B^\#_\xi)
 =\langle\xi,[A,B]\rangle.
\]
This is the orbit symplectic form for the Lie--Poisson structure \cite{Kirillov}; right equivariance also reads 
\[\mu(p\cdot g)=\Ad_{g^{-1}}^*\mu(p).\]

\begin{theorem}[Symplectic leaves of the logahoric quotient]\label{thm:allleaves}
For every coadjoint orbit $\mathscr O\subset\gD^*$, the geometric quotient
\[
 \cL_{\mathscr O}=\mu^{-1}(\mathscr O)/\GD
\]
exists, is nonempty, smooth, and connected, and is a symplectic leaf of $\cP$. These are all the symplectic leaves. If $i:\mu^{-1}(\mathscr O)\hookrightarrow\cX$ is the inclusion and $q:\mu^{-1}(\mathscr O)\to\cL_{\mathscr O}$ is the quotient map, its symplectic form is characterized by
\begin{equation}\label{eq:reducedform}
 q^*\Omega_{\mathscr O}
 =i^*\omega_{LH}+\mu^*\omega_{\mathscr O}.
\end{equation}
Moreover,
\begin{equation}\label{eq:leafdimension}
 \dim\cL_{\mathscr O}=2\dim\cB+\dim\mathscr O.
\end{equation}
\end{theorem}

\begin{proof}
Lemma~\ref{lem:musubmersion} makes $\mu^{-1}(\mathscr O)$ smooth and nonempty. It is an affine bundle over $\cU\times\mathscr O$, which is connected because $\cB$ and $\GD$ are connected. Its quotient is therefore connected. Proposition~\ref{prop:cotangentquotientcharts} gives the smooth quotient charts $T^*V\times\mathscr O$ and the principal $\GD$-bundle $q$.

A fixed $\xi\in\mathscr O$ has stabilizer $\GD_\xi$. Equivariance gives
\[
 \mu^{-1}(\xi)\times^{\GD_\xi}\GD
 \longrightarrow\mu^{-1}(\mathscr O),\qquad[p,g]\longmapsto p\cdot g,
\]
where $(p,g)\sim(p\cdot k,k^{-1}g)$, for $k\in\GD_\xi$. After the faithfully flat base change $\GD\to\mathscr O$, with $g\mapsto\Ad_{g^{-1}}^*\xi$, this map becomes the identity on $\mu^{-1}(\xi)\times\GD$, hence is an isomorphism by descent. Consequently, quotienting by the residual right action gives $\cL_{\mathscr O}\cong\mu^{-1}(\xi)/\GD_\xi$.

The closed invariant form $\beta=i^*\omega_{LH}+\mu^*\omega_{\mathscr O}$ satisfies
\[
 \iota_{A^\#}i^*\omega_{LH}=d\langle\mu,A\rangle,
 \qquad
 \iota_{A^\#}\mu^*\omega_{\mathscr O}=-d\langle\mu,A\rangle,
\]
by the moment identity and right equivariance: $d\mu(A^\#)$ is the negative of the coadjoint fundamental vector at $\mu(p)$. The contractions cancel, so $\beta$ vanishes along the fibers of $q$. Its invariance then makes it descend as in \eqref{eq:reducedform}.

At $p\in\mu^{-1}(\mathscr O)$, the tangent space to the group orbit is $V_p=T_p(p\cdot\GD)$. The moment identity gives $\ker(d\mu_p)=V_p^{\omega_{LH}}$, and the preceding contractions give $V_p\subset\ker\beta_p$. If $v\in\ker\beta_p$, then $\mu^*\omega_{\mathscr O}(v,w)=0$, for every $w\in\ker(d\mu_p)$. Thus $\omega_{LH}(v,w)=0$ for all such $w$, so
\[
 v\in(V_p^{\omega_{LH}})^{\omega_{LH}}=V_p.
\]
Hence $\ker\beta_p=V_p$ and the descended form is nondegenerate.

In the Poisson charts $T^*V\times\gD^*$, Hamiltonian vector fields span exactly $T(T^*V\times\mathscr O)$. Transitions conjugate the residue, so these products glue to $\cL_{\mathscr O}$ and every leaf stays in one orbit reduction. Thus each connected reduction has the full tangent distribution and hence is one leaf. The orbit reductions cover $\cP$.

Finally, smoothness of $\mu$ and freeness of the action give
\[
 \dim\cL_{\mathscr O}
 =2\dim\cU-(\dim\GD-\dim\mathscr O)-\dim\GD
 =2\dim\cB+\dim\mathscr O.
\]
\end{proof}

\subsection{From sheets to loci of fixed Poisson rank}

The Poisson tensor on $\cP$ is denoted by $\Pi$. Its rank at $y$ is the rank of $\Pi_y^\#:T_y^*\cP\to T_y\cP$.

\begin{lemma}[Irreducibility from open fibers]\label{lem:openirreducible}
The map $f:X\to S$ is an open surjection of topological spaces. If $S$ and every fiber $f^{-1}(s)$ are irreducible, then $X$ is irreducible.
\end{lemma}

\begin{proof}
For nonempty open $U,V\subset X$, openness and irreducibility give $s\in f(U)\cap f(V)$. The nonempty open subsets $U\cap f^{-1}(s)$ and $V\cap f^{-1}(s)$ of the irreducible fiber meet. Thus $U\cap V\ne\varnothing$.
\end{proof}

\begin{theorem}[Sheet loci in the Poisson moduli]\label{thm:sheetstrata}
A sheet $S\subset(\gD^*)_{(d)}$ defines the locus
\[
 \cP_S=\mu^{-1}(S)/\GD\subset\cP.
\]
The nonempty loci $\cP_S$, with $S$ ranging over the sheets in $(\gD^*)_{(d)}$, are exactly the irreducible components of
\[
 \{y\in\cP:\operatorname{rank}\Pi_y=2\dim\cB+d\},
\]
with its reduced locally closed algebraic-space structure. The leaves in $\cP_S$ are $\cL_{\mathscr O}$ for the coadjoint orbits $\mathscr O\subset S$.

If $J=J_{\GD}(\lalg,\cO)$ is the dense Jordan class of $S$, there is a canonical smooth morphism
\begin{equation}\label{eq:globalJordanparameter}
 \kappa_J:\mu^{-1}(J)/\GD
 \longrightarrow\zM^{\reg}/Q_M.
\end{equation}
For $m\in\zM^{\reg}$ and $e\in\cO$, its scheme-theoretic fiber over $[m]$ is $\cL_{\GD\cdot(m+e)}$. These fibers are the leaves whose coresidue lies in $J$.
\end{theorem}

\begin{proof}
On an \'etale chart trivializing the level bundle, Proposition~\ref{prop:cotangentquotientcharts} gives
\[
 \cP|_V\simeq T^*V\times\gD^*,\qquad
 \cP_S|_V\simeq T^*V\times S.
\]
The cotangent factor has Poisson rank $2\dim\cB$, and the residue factor has rank $\dim\GD\cdot\xi$ at $\xi$. Hence the rank locus is $\mu^{-1}((\gD^*)_{(d)})/\GD$. Coadjoint-invariance of $S$ makes the products descend.

Note that the connected smooth base $\cB$ is irreducible, since its irreducible components are disjoint open subsets. The projection $\cP_S\to\cB$ is surjective, flat and locally of finite presentation by these charts, hence open. Its geometric fibers are affine spaces times the geometrically irreducible sheet $S$, so Lemma~\ref{lem:openirreducible} proves irreducibility.

Sheets are closed irreducible components of $(\gD^*)_{(d)}$. Their descended loci therefore cover the rank locus by closed irreducible subsets. Since locally they are the distinct, mutually noncontained products $T^*V\times S$, they are precisely its irreducible components. Theorem~\ref{thm:allleaves} identifies their leaves.

The invariant composition $\pi_J\circ\mu$, with $\pi_J$ the geometric orbit quotient of Proposition~\ref{prop:jordanclassquotient}, descends to $\kappa_J$. Locally it sends $(\alpha,\xi)\in T^*V\times J$ to $\pi_J(\xi)$, hence is smooth. The quotient charts commute with base change, giving
\[
 \kappa_J^{-1}([m])
 =\mu^{-1}\!\bigl(\pi_J^{-1}([m])\bigr)/\GD
 =\mu^{-1}\!\bigl(\GD\cdot(m+e)\bigr)/\GD
 =\cL_{\GD\cdot(m+e)}
\]
scheme-theoretically.
\end{proof}

\section{The Simultaneous Parabolic Reduction}\label{sec:simultaneous}

For this section, the parabolic $P_M=MN$ and nilpotent $M$-orbit $\cO$ remain fixed as in Section~\ref{sec:leviresidues}. The residue condition is $\mu\in m+\cO+\nalg$. We impose it in two stages: reduction by $N$ at zero places the residue in $\palg$, and the subsequent $M$-reduction fixes its Levi component to $m+\cO$.

\subsection{Parabolic reductions and their cotangent bundles}

A point of $\cU/P_M$ retains a $P_M$-reduction of the residue torsor, while a point of $\cU/N$ retains a frame modulo $N$. The corresponding spaces $\cBP=\cU/P_M$ and $\cV=\cU/N$ satisfy
\[
 \cV=\cU\times^{\GD}(\GD/N),\qquad
 \cBP=\cU\times^{\GD}(\GD/P_M),\qquad
 \cV/M=\cBP,
\]
and are smooth separated algebraic spaces of finite type: over an \'etale chart trivializing $\cU$, they are $W\times\GD/N$ and $W\times\GD/P_M$. Thus $\cBP$ parametrizes $P_M$-reductions of residue torsors, and $\cV\to\cBP$ is a principal $M=P_M/N$-bundle, with
\[
 \dim\cBP=\dim\cB+\dim N
 =\dim\cB+\frac12(\dim\gD-\dim\lalg).
\]

The restriction of $\mu$ to $\nalg$ is $\mu_N:T^*\cU\to\nalg^*$.
\begin{lemma}[Cotangent reduction by the unipotent radical]\label{lem:cotangentN}
There is a canonical symplectomorphism
\[
 \mu_N^{-1}(0)/N\cong T^*\cV.
\]
The induced $M=P_M/N$-action is the cotangent lift of the action on $\cV$. Its moment map $\mu_M:T^*\cV\to\lalg^*$ is the Levi component of the coresidue $\mu$.
\end{lemma}
\begin{proof}
The $N$-torsor $r_N:\cU\to\cV$ has tangent sequence
\[
 0\longrightarrow\cU\times\nalg
 \longrightarrow T\cU
 \xrightarrow{\,dr_N\,}r_N^*T\cV
 \longrightarrow0.
\]
The moment map evaluates covectors on the tangent directions of the $N$-action. Thus $\mu_N=0$ precisely when, for $v=r_N(u)$,
\[
 \alpha_u=(dr_N)_u^*\beta_v,
 \qquad
 \mu_N^{-1}(0)\cong\cU\times_{\cV}T^*\cV.
\]
Quotienting this $N$-torsor therefore gives $T^*\cV$. For the inclusion $i$ of the zero level and its projection $r$ to $T^*\cV$, covector evaluation gives $i^*\lambda_{\cU}=r^*\lambda_{\cV}$; applying $-d$ then proves symplecticity.

The $P_M$-equivariant identification induces the cotangent lift of $(uN)\cdot l=(ul)N$. For $A\in\lalg$,
\[
 \langle\mu_M(v,\beta_v),A\rangle
 =\beta_v(A^\#(v))
 =\alpha_u(A^\#(u))
 =\langle\mu(u,\alpha_u),A\rangle.
\]
Since $\mu_N=0$ and $N$ acts trivially on $\palg/\nalg$, restriction of $\mu$ to $\palg$ descends to $(\palg/\nalg)^*=\lalg^*$.

The residue corresponding to the moment-map value is $y=B_{\gD}^{-1}\mu(u,\alpha_u)$. Since $\nalg^\perp=\palg$ and $B_{\gD}(\nalg,\palg)=0$, the zero condition gives $y=y_{\lalg}+y_{\nalg}$ and
\[
 \langle\mu_M(v,\beta_v),A\rangle
 =B_{\gD}(y_{\lalg},A).
\]
Thus nondegeneracy of $B_{\gD}|_{\lalg}$ identifies $\mu_M$ with the Levi residue $y_{\lalg}$.
\end{proof}

\subsection{The family over the center}

Note that the decomposition $\zM+\cO\cong\zM\times\cO$ has an $M$-invariant projection.
\begin{definition}[Simultaneous parabolic reduction]\label{def:simultaneous}
The \emph{simultaneous parabolic reduction} is defined by
\[
 \widetilde\cL=\mu_M^{-1}(\zM+\cO)/M.
\]
The projection $\operatorname{pr}_{\zM}\circ\mu_M$ then descends to $q:\widetilde\cL\to\zM$, whose fiber at $m$ is
\[
 \widetilde\cL_m=\mu_M^{-1}(m+\cO)/M.
\]
\end{definition}

\begin{proposition}[Existence of the smooth family]\label{prop:smoothfamily}
The space $\widetilde\cL$ is smooth, and $q:\widetilde\cL\to\zM$ is smooth and surjective. Each fiber is nonempty and holomorphic symplectic, with reduced form $\Omega_m$ and dimension
\begin{equation}\label{eq:simultaneousdimension}
 \dim\widetilde\cL_m
 =2\dim\cB+\dim\gD-\dim\lalg+\dim\cO.
\end{equation}
\end{proposition}
\begin{proof}
For the cotangent projection $p:T^*\cV\to\cV$, the map $(p,\mu_M):T^*\cV\to\cV\times\lalg^*$ is a vector-bundle epimorphism, so $\mu_M$ is smooth and surjective. Proposition~\ref{prop:cotangentquotientcharts} gives, on an \'etale chart trivializing $\cV\to\cBP$,
\[
 \widetilde\cL|_W\cong T^*W\times\zM\times\cO,
 \qquad q(\beta,m,e)=m.
\]
The product charts therefore prove smoothness and surjectivity. The orbit-reduction argument of Theorem~\ref{thm:allleaves}, applied to the free $M$-action, gives the fiber symplectic forms and
\[
 \dim\widetilde\cL_m=2\dim\cBP+\dim\cO
 =2\dim\cB+\dim\gD-\dim\lalg+\dim\cO.
\]
\end{proof}

\begin{lemma}[The space with a parabolic reduction]\label{lem:comparisonspaces}
For every $m\in\zM$,
\[
 \widetilde\cL_m
 \cong\mu^{-1}(m+\cO+\nalg)/P_M
 \cong\bigl(\cX\times_{\gD^*}
           \widetilde\cJ_m\bigr)/\GD.
\]
In the fiber product, a point is a leveled Higgs field together with a compatible parabolic reduction of its residue.
\end{lemma}
\begin{proof}
Lemma~\ref{lem:cotangentN} gives the first quotient: the zero $N$-moment condition says that the residue lies in $\palg$, and the $M$-moment condition fixes its Levi component. For the second identification, the action is $(x,j)\cdot h=(x\cdot h,h^{-1}j)$. A representative $(x,[h;y])$ is sent to $(x\cdot h,[1;y])$, with $\mu(x\cdot h)=y$. Two representatives with second coordinate $[1;y]$ differ precisely by $P_M$, giving the quotient $\mu^{-1}(m+\cO+\nalg)/P_M$.
\end{proof}

\begin{lemma}[Restriction of the orbit form]\label{lem:parabolicforms}
The residue orbits are $\mathscr O_m=\GD\cdot(m+e)$ for regular $m$, and $\mathscr O_0=\Ind_{\lalg}^{\gD}(\cO)$. The parabolic residue conditions define
\[
 S_m=m+\cO+\nalg\quad(m\in\zM^{\reg}),\qquad
 S_0=(\cO+\nalg)\cap\mathscr O_0.
\]
The map $\operatorname{pr}_{\cO}:S_m\to\cO$, $m+e'+n_y\mapsto e'$, satisfies
\[
 \omega_{\mathscr O_m}|_{S_m}=\operatorname{pr}_{\cO}^*\omega_{\cO}.
\]
\end{lemma}
\begin{proof}
Lemmas~\ref{lem:unipotentorbit} and~\ref{lem:inducedorbitcover} make $S_m$ a single $P_M$-orbit, so its tangent vectors are $[A,y]$, $A\in\palg$. The residue decomposes as $y=m+e'+n_y$, and $A_{\lalg}=\operatorname{pr}_{\lalg}(A)$ is the Levi component. For $A,B\in\palg$,
\[
 \begin{aligned}
 (\omega_{\mathscr O_m})_y([A,y],[B,y])
 &=B_{\gD}(y,[A,B])\\
 &=B_{\gD}(e',[A_{\lalg},B_{\lalg}])\\
 &=(\operatorname{pr}_{\cO}^*\omega_{\cO})_y([A,y],[B,y]).
 \end{aligned}
\]
The middle equality uses $B_{\gD}(\palg,\nalg)=0$ and the fact that $m$ annihilates $[\lalg,\lalg]$.
\end{proof}

\begin{theorem}[Comparison with the residue-orbit reductions]\label{thm:simultaneous}
For $m\in\zM^{\reg}$ and $e\in\cO$, forgetting the parabolic reduction gives a symplectomorphism
\[
 \widetilde\cL_m\xrightarrow{\sim}\cL_{\GD\cdot(m+e)}.
\]
This fiber is a symplectic leaf of $\cP$. At $m=0$, the induced orbit $\mathscr O_0=\Ind_{\lalg}^{\gD}(\cO)$ determines the open subset
\[
 \widetilde\cL_0^\circ
 =\mu^{-1}\bigl((\cO+\nalg)\cap\mathscr O_0\bigr)/P_M
 \subset\widetilde\cL_0.
\]
Forgetting the parabolic reduction gives a finite \'etale symplectic morphism
\[
 f_0:\widetilde\cL_0^\circ\longrightarrow\cL_{\mathscr O_0}
\]
of generalized Springer degree. It is an isomorphism when the degree is one.
\end{theorem}
\begin{proof}
Lemma~\ref{lem:comparisonspaces} expresses the fibers as quotients of the local residue spaces. For regular $m$, Theorem~\ref{thm:parabolicfibers} identifies $\widetilde\cJ_m$ with $\GD\cdot(m+e)$; base change by $\mu$ and the free quotient therefore give the asserted isomorphism. At zero, the finite \'etale map $\nu_0:\widetilde{\mathscr O}_0\to\mathscr O_0$ of Lemma~\ref{lem:inducedorbitcover} gives
\[
 \widetilde\cL_0^\circ
 \cong\bigl(\cX\times_{\gD^*}\widetilde{\mathscr O}_0\bigr)/\GD,
 \qquad f_0:\widetilde\cL_0^\circ\to\cL_{\mathscr O_0}.
\]
Pulling $f_0$ back along $\mu^{-1}(\mathscr O_0)\to\cL_{\mathscr O_0}$ gives the projection $\mu^{-1}(\mathscr O_0)\times_{\mathscr O_0}\widetilde{\mathscr O}_0\to\mu^{-1}(\mathscr O_0)$. This is the base change of $\nu_0$, so it is finite \'etale of the same degree. These properties descend along the torsor.

Lemma~\ref{lem:parabolicforms} identifies the orbit term in the reduction form \eqref{eq:reducedform} with the Levi orbit term. The cotangent terms agree by Lemma~\ref{lem:cotangentN}. Both terms therefore agree before quotienting, giving the asserted symplectomorphism and $f_0^*\Omega_{\mathscr O_0}=\Omega_0|_{\widetilde\cL_0^\circ}$. Theorem~\ref{thm:allleaves} identifies the leaves.
\end{proof}

In particular, the orbit dimension formula of Proposition~\ref{prop:induceddimension} and Theorem~\ref{thm:allleaves} give, for $m\in\zM^{\reg}$,
\begin{equation}\label{eq:globaldimensioncondition}
 \dim\cL_{\GD\cdot(m+e)}
 =2\dim\cB+\dim\gD-\dim\lalg+\dim\cO
 =\dim\cL_{\Ind_{\lalg}^{\gD}(\cO)}.
\end{equation}
\begin{example}[The zero fiber for a torus Levi]\label{ex:SL2globalfamily}
In this example, $\gD=\mathfrak{sl}_2$, $M$ is the diagonal torus, and $\cO=0$. Then $\cBP$ parametrizes residue flags and
\[
 \widetilde\cL_0=T^*\cBP.
\]
On a level chart, $\cBP|_W=W\times\mathbb P^1$, so
\[
 \widetilde\cL_0|_W=T^*W\times T^*\mathbb P^1.
\]
The Springer map sends $(\ell,x)\in T^*\mathbb P^1$, with $x(\C^2)\subset\ell\subset\ker x$, to $x$. For $x\ne0$, $\ell=\operatorname{im}x=\ker x$ is unique; over $x=0$ its fiber is the zero section $\mathbb P^1$. Thus $T^*W\times(T^*\mathbb P^1\setminus\mathbb P^1)$ maps isomorphically to the nonzero nilpotent orbit reduction.
\end{example}

\subsection{Orbit closures and the forgetful morphism}\label{subsec:orbitclosures}

The reduced orbit closure $\overline\cO$ defines the family
\[
 \overline{\widetilde\cL}
 :=\mu_M^{-1}(\zM+\overline\cO)/M,
 \qquad
 \overline q:\overline{\widetilde\cL}\longrightarrow\zM.
\]
Here $\overline q$ is the projection, and $\widetilde\cL$ is the open subset with Levi nilpotent residue in $\cO$.

\begin{proposition}[The flat family with orbit-closure fibers]\label{prop:orbitclosurefamily}
The morphism $\overline q$ is flat, surjective, and of finite presentation. Its scheme-theoretic fiber at $m\in\zM$ is
\[
 \overline{\widetilde\cL}_m
 =\mu_M^{-1}(m+\overline\cO)/M.
\]
Every geometric fiber is reduced and has dimension
\[
 2\dim\cB+\dim\gD-\dim\lalg+\dim\cO.
\]
The total space has a Poisson bracket for which the central coordinates are \emph{Casimirs}, meaning that their brackets with all functions vanish. This bracket restricts to every fiber.
\end{proposition}
\begin{proof}
Since $\zM+\overline\cO=\zM\times\overline\cO$ is closed, Proposition~\ref{prop:cotangentquotientcharts} gives, on an \'etale chart trivializing $\cV\to\cBP$,
\[
 \overline{\widetilde\cL}|_W
 \cong T^*W\times\zM\times\overline\cO,
 \qquad \overline q(\beta,m,e)=m.
\]
These charts commute with base change and therefore prove flatness, surjectivity, finite presentation and the fiber formula. The geometric fibers are reduced because $\overline\cO$ is geometrically reduced over $\C$ and $T^*W$ is smooth; their dimension is $2\dim\cBP+\dim\cO$.

The $M$-invariant ideal of $\overline\cO$ is preserved by brackets with linear functions, which differentiate the adjoint action, and hence with all polynomials. Thus $\zM+\overline\cO\subset\lalg^*$ is Poisson, with Casimir coordinates. The chart bracket is the product of the cotangent, zero, and restricted Lie--Poisson brackets. These restrictions of the bracket on $T^*\cV/M$ agree on overlaps, and every fiber ideal is Poisson.
\end{proof}

\begin{proposition}[The proper forgetful morphism]\label{prop:propercomparison}
Forgetting the parabolic reduction defines a proper Poisson morphism
\begin{equation}\label{eq:properclosurecomparison}
 \overline f:\overline{\widetilde\cL}
 \longrightarrow\cP\times\zM.
\end{equation}
\end{proposition}
\begin{proof}
Replacing $\cO$ by $\overline\cO$ gives the space of residues with a parabolic reduction:
\[
 \overline{\widetilde\cJ}
 :=\GD\times^{P_M}
 \{(m,y)\in\zM\times\palg:
     \operatorname{pr}_{\lalg}(y)\in m+\overline\cO\}.
\]
Its closed embedding $[h;(m,y)]\mapsto(hP_M,m,\Ad(h)y)$ into $\GD/P_M\times\zM\times\gD$ and projectivity of $\GD/P_M$ make
\[
 \overline\nu:\overline{\widetilde\cJ}
 \longrightarrow\zM\times\gD
\]
proper. The proof of Lemma~\ref{lem:comparisonspaces}, with $\overline\cO$ in place of $\cO$, then gives
\[
 \overline{\widetilde\cL}
 \cong
 \bigl((\cX\times\zM)
       \times_{\gD^*\times\zM}\overline{\widetilde\cJ}\bigr)/\GD.
\]
The invariant form identifies $\gD$ with $\gD^*$. After exchanging the target coordinates, base change of $\overline\nu$ and descent along the $\GD$-torsor $\cX\times\zM\to\cP\times\zM$ give the proper morphism $\overline f$.

A function from $\cP$ lifts to a $\GD$-invariant function on $\cX$, commuting with $\mu_N$. Its restriction descends through cotangent reduction by $N$, then through the $M$-quotient and the Poisson condition $\zM+\overline\cO$, preserving brackets. The central coordinates pull back to Casimirs, so $\overline f$ is Poisson.
\end{proof}

\begin{proposition}[Regular and zero orbit-closure fibers]\label{prop:closurecomparison}
For $m\in\zM^{\reg}$, $\overline f$ identifies its fiber scheme-theoretically with the reduced closed subspace
\begin{equation}\label{eq:regularclosurecomparison}
 \overline{\widetilde\cL}_m
 \xrightarrow{\sim}
 \mu^{-1}\bigl(\GD\cdot(m+\overline\cO)\bigr)/\GD.
\end{equation}
At zero it factors through a proper surjection
\begin{equation}\label{eq:zeroclosurecomparison}
 \overline f_0:\overline{\widetilde\cL}_0
 \longrightarrow\mu^{-1}(\overline{\mathscr O_0})/\GD,
 \qquad \mathscr O_0=\Ind_{\lalg}^{\gD}(\cO),
\end{equation}
of generalized Springer degree over the dense orbit reduction. Its inverse image of $\cL_{\mathscr O_0}$ is $\widetilde\cL_0^\circ$, and the restriction is the finite \'etale symplectic morphism $f_0$ of Theorem~\ref{thm:simultaneous}.
\end{proposition}
\begin{proof}
For regular $m$, Proposition~\ref{prop:residualsingularities} identifies $\overline{\widetilde\cJ}_m$ with its reduced closed image $\GD\cdot(m+\overline\cO)$. Thus smooth base change by $\mu$ and torsor descent give \eqref{eq:regularclosurecomparison}.

At $m=0$, the finite-dimensional map is the proper generalized Springer morphism
\[
 \GD\times^{P_M}(\overline\cO+\nalg)
 \longrightarrow\overline{\mathscr O_0}.
\]
The scheme-theoretic image is the smallest closed subscheme through which the morphism factors. Here it is $\overline{\mathscr O_0}$ with reduced structure: the source is reduced and its image contains $\mathscr O_0$. Scheme-theoretic images commute with flat base change \cite[Tag~081I]{Stacks}. Thus smooth base change by $\mu$ and torsor descent identify the target of \eqref{eq:zeroclosurecomparison} with the scheme-theoretic image of this fiber map. Its cotangent quotient charts show it is reduced.

By the dimension inequality of Lemma~\ref{lem:inducedorbitcover}, smaller Levi orbits cannot map into $\mathscr O_0$, whose inverse image is $\GD\times^{P_M}((\cO+\nalg)\cap\mathscr O_0)$. Consequently, base change and descent give $\widetilde\cL_0^\circ$ and its finite \'etale map of generalized Springer degree. Symplecticity follows from Theorem~\ref{thm:simultaneous}.
\end{proof}

The scheme-theoretic image $\mathcal D\subset\cP\times\zM$ is reduced because the source is reduced, and Poisson because $\overline f$ preserves brackets. Properness identifies its fiber supports with the fiber images; hence
\begin{equation}\label{eq:reducedimagezero}
 (\mathcal D_0)_{\mathrm{red}}
 =\mu^{-1}(\overline{\mathscr O_0})/\GD.
\end{equation}
\section{Variation of the Symplectic Class}\label{sec:periods}

In this section, we study the periods of the local residue family. In general, a connection on $\cV\to\cBP$ identifies the fibers of $q$ with a common differentiable manifold. Under this identification, the difference of their symplectic forms is minus the pullback of the curvature paired with the change of central parameter and its cohomology class is the Duistermaat--Heckman variation \cite{DH}.

Recall that the form $B_{\gD}$ restricted on $\lalg$ identifies $\lalg$ with $\lalg^*$.

\subsection{Covectors and a common differentiable manifold}

The principal $M$-bundle has projection $\pi_{P_M}:\cV\to\cBP$. Dividing the tangent sequence by $M$ gives
\[
 0\longrightarrow\cV\times_M\lalg
 \longrightarrow T\cV/M
 \longrightarrow T\cBP\longrightarrow0.
\]
The first term consists of tangent directions along the $M$-fibers. Since $\cBP$ is separated and of finite type, its analytic manifold admits partitions of unity. These split the sequence smoothly and complex-linearly. Pulling the splitting back to $\cV$ gives an $M$-equivariant connection
\[
 A\in\Omega^1_{C^\infty}(\cV)\otimes\lalg,
 \qquad A(X^\#)=X,
\]
for $X\in\lalg$, where $X^\#$ denotes the right fundamental vector field. For the complex structure $J$, it is $A(Jw)=iA(w)$. The horizontal bundle $\ker A$ complements the vertical bundle $\ker d\pi_{P_M}$.

The horizontal, equivariant curvature $F_A=dA+\tfrac12[A,A]$ descends to
\[
 F_A\in\Omega^2\bigl(\cBP,\cV\times_M\lalg\bigr).
\]
Its pairing with the invariant covector $B_{\gD}(m,-)$ is a complex two-form on $\cBP$.

A covector at $v\in\cV$ has a horizontal part $\beta\in T^*_{\pi_{P_M}(v)}\cBP$ and a vertical part $\xi\in\lalg^*$. Restriction to these directions therefore gives an $M$-equivariant smooth complex vector-bundle isomorphism
\begin{equation}\label{eq:cotangentsplitting}
 T^*\cV
 \cong
 \cV\times_{\cBP}T^*\cBP\times\lalg^*,
\end{equation}
with inverse
\begin{equation}\label{eq:inversecovector}
 \alpha_v(w)=\beta(d\pi_{P_M}(w))+\langle\xi,A(w)\rangle.
\end{equation}

The cotangent action is
\[
 (v,\beta,\xi)\cdot l
 =(vl,\beta,\Ad^*_{l^{-1}}\xi),\qquad l\in M.
\]
Lastly, evaluation on $X^\#$ identifies $\mu_M$ with $\xi$.

The condition $\mu_M\in m+\cO$ fixes $\xi=B_{\gD}(m+e,-)$, for $e\in\cO$, while $\beta$ remains free. Subtracting $m$ from the vertical coordinate, leaves the same set of triples $(v,\beta,e)$ for every parameter.

\begin{proposition}[Differentiable identification of the fibers]\label{prop:commonmodel}
The connection identifies every simultaneous reduction $\widetilde\cL_m$ with the same $C^\infty$ manifold
\begin{equation}\label{eq:Ymodel}
 Y
 =T^*\cBP\times_{\cBP}(\cV\times_M\cO).
\end{equation}
Under these identifications the family $\widetilde\cL\to\zM$ is globally $C^\infty$ trivial.
\end{proposition}

\begin{proof}
For $\beta\in T^*_{\pi_{P_M}(v)}\cBP$ and $e\in\cO$, formula~\eqref{eq:inversecovector} defines $\alpha_v$ with vertical covector $B_{\gD}(m+e,-)$. The map $[v,\beta,e]\mapsto[v,\alpha_v]$ is therefore a map from $Y$ to $\widetilde\cL_m$. Its inverse splits a covector as in \eqref{eq:cotangentsplitting} and subtracts $m$ from its vertical coordinate. Both constructions are $M$-equivariant and smooth in $m$, so they give $Y\times\zM\cong\widetilde\cL$ over $\zM$.
\end{proof}

In the sequel, the pulled-back reduced form on $Y$ will be denoted by $\Omega_m$, and $\pi_Y:Y\to\cBP$ is the projection.

\subsection{The reduced form and exactness at zero}

Covector evaluation in \eqref{eq:inversecovector} gives, on the splitting space,
\begin{equation}\label{eq:liouvillesplitting}
 \lambda_{\cV}=\lambda_{\cBP}+\langle\xi,A\rangle.
\end{equation}
Since $\omega=-d\lambda$,
\begin{equation}\label{eq:canonicalsplitting}
 \omega_{T^*\cV}
 =\omega_{T^*\cBP}-d\langle\xi,A\rangle.
\end{equation}

Translation by $m$ preserves the KKS form $\omega_{\cO}$ because $B_{\gD}(m,[X_1,X_2])=0$, for $X_i\in\lalg$. Thus, the form defining reduction in \eqref{eq:reducedform}, before the $M$-quotient, is
\begin{equation}\label{eq:minimalcoupling}
 \omega_{T^*\cBP}
 -d\langle B_{\gD}(m+e,-),A\rangle
 +\omega_{\cO}(e),\qquad e\in\cO.
\end{equation}
The moment-map and KKS identities make this form \emph{basic}, that is, it is invariant and vanishes whenever one argument is tangent to an $M$-orbit. It therefore descends to $\Omega_m$ on $Y$.

\begin{lemma}[Exactness at a nilpotent value]\label{lem:exactnilpotent}
The reduced symplectic forms on $\widetilde\cL_0$ and on $\cL_{\mathscr O_0}$, where $\mathscr O_0=\Ind_{\lalg}^{\gD}(\cO)$, are exact. In particular, $[\Omega_0]=0$ on the common manifold $Y$.
\end{lemma}

\begin{proof}
For a nilpotent $e$, the Jacobson--Morozov theorem gives a homomorphism $\lambda:\C^*\to M$ with $\Ad(\lambda(t))e=t^2e$ \cite{CM}. Every $t\in\C^*$ is a square, so $te$ lies in the same $M$-orbit as $e$. Thus every nilpotent orbit is preserved by $\C^*$-scaling. Consequently, the dilation $(v,\alpha)\mapsto(v,t\alpha)$ preserves $\mu_M^{-1}(\cO)$ and descends to $\widetilde\cL_0$. It multiplies the cotangent form by $t$, and the KKS form scales by the same factor because
\[
 B_{\gD}(te,[X_1,X_2])
 =tB_{\gD}(e,[X_1,X_2]).
\]
Hence, $t^*\Omega_0=t\Omega_0$. For the infinitesimal generator $E$ of this scaling action, Cartan's formula gives
\[
 \Omega_0=\mathcal L_E\Omega_0=d(\iota_E\Omega_0).
\]
The same scaling argument applies to the nilpotent orbit reduction $\cL_{\mathscr O_0}$.
\end{proof}

\subsection{Curvature and the period formula}

\begin{theorem}[Period variation]\label{thm:periodvariation}
Under the identifications of Proposition~\ref{prop:commonmodel},
\begin{equation}\label{eq:formvariation}
 \Omega_m-\Omega_{m_0}
 =-\pi_Y^*\langle B_{\gD}(m-m_0,-),F_A\rangle
\end{equation}
as closed complex-valued smooth two-forms on $Y$. Consequently,
\begin{equation}\label{eq:periodvariation}
 [\Omega_m]-[\Omega_{m_0}]
 =\langle m-m_0,c\rangle,
 \qquad m,m_0\in\zM,
\end{equation}
where $c\in H^2(Y,\C)\otimes\zM^*$ is the linear map characterized by
\[
 \langle m,c\rangle
 =-[\pi_Y^*\langle B_{\gD}(m,-),F_A\rangle].
\]
The class $c$ is independent of the connection, and $[\Omega_m]=\langle m,c\rangle$.
\end{theorem}

\begin{proof}
Subtracting \eqref{eq:minimalcoupling} at $m_0$ from its value at $m$ gives $-d\langle\delta,A\rangle$, where $\delta=B_{\gD}(m-m_0,-)|_{\lalg}$.

We observe that invariance of $\delta$ and equivariance of $A$ give
\[
 \iota_{X^\#}d\langle\delta,A\rangle
 =\mathcal L_{X^\#}\langle\delta,A\rangle
 -d\langle\delta,X\rangle=0.
\]
Thus $d\langle\delta,A\rangle$ descends to $\cBP$.

Since $\delta([\lalg,\lalg])=0$, then 
\[
 \langle\delta,F_A\rangle
 =\langle\delta,dA+\tfrac12[A,A]\rangle
 =d\langle\delta,A\rangle.
\]
For the quotient map $r$ from the space of triples $(v,\beta,e)$ to $Y$,
\[
 r^*(\Omega_m-\Omega_{m_0})
 =-d\langle\delta,A\rangle
 =-r^*\pi_Y^*\langle\delta,F_A\rangle.
\]
Since $r$ is a surjective submersion, pullback is injective on forms, giving \eqref{eq:formvariation}. Taking cohomology classes and using $[\Omega_0]=0$ from Lemma~\ref{lem:exactnilpotent} gives the period formulas.

For another connection $A'$, the difference $a=A'-A$ is horizontal and equivariant, and
\[
 \langle\delta,F_{A'}-F_A\rangle=d\langle\delta,a\rangle.
\]
The invariant form $\langle\delta,a\rangle$ descends to $\cBP$, so the curvature classes differ by an exact form there. The path $(1-s)A+sA'$ gives an isotopy of the fiber identifications, hence the same identification in cohomology.
\end{proof}

\begin{example}[A one-parameter cotangent family]
\label{ex:oneparameterperiods}
For $\cO=0$ and $\zM=\C z$, $Y=T^*\cBP$ and $\Omega_0$ is canonical. Setting $m=az$ and $m_0=a_0z$ gives
\[
 \Omega_a-\Omega_{a_0}
 =-(a-a_0)\,\pi^*\langle B_{\gD}(z,-),F_A\rangle,
 \qquad \pi:T^*\cBP\to\cBP.
\]
For a two-cycle $\Sigma$ in the zero section,
\[
 \int_\Sigma(\Omega_a-\Omega_{a_0})
 =-(a-a_0)\int_\Sigma\langle B_{\gD}(z,-),F_A\rangle.
\]
In Example~\ref{ex:SL4running}, the central generator is $z=\operatorname{diag}(I_2,-I_2)$. For $a\ne0$ these are orbit reductions for regular parameters; at zero, the open subset of Theorem~\ref{thm:simultaneous} is the degree-one cover of the $[2^2]$ reduction.
\end{example}

\subsection{The characteristic class of the central torus}

The connected derived subgroup $M_{\der}=[M,M]$ has torus quotient
\[
 T_M=M/M_{\der}.
\]
The quotient map $Z(M)^\circ\to T_M$ is surjective with finite kernel, so it identifies $\Lie(T_M)$ with $\zM$. The form $B_{\gD}$ identifies $\zM$ with its dual and makes it orthogonal to $[\lalg,\lalg]$.

Quotienting $\cV\times\cO$ by $M_{\der}$ and then by $T_M$ gives
\[
 \cV\times_{M_{\der}}\cO
 \longrightarrow\cV\times_M\cO.
\]
The residual action is free: the fact that $vl=vd$, for $d\in M_{\der}$, forces $l=d$. Pullback to $Y$ therefore gives the principal $T_M$-bundle
\[
 \mathscr P
 =T^*\cBP\times_{\cBP}
 (\cV\times_{M_{\der}}\cO)
 \longrightarrow Y.
\]
Projection of $A$ to $\Lie(T_M)$ induces a connection with curvature $\pi_Y^*\overline F_A$, the projected curvature of $A$.

The integral characteristic class in $H^2(Y,X_*(T_M))$ pairs with each character $\chi:T_M\to\mathbb G_m$ to give the first Chern class $c_1(L_\chi)$ of
\[
 L_\chi=(\mathscr P\times\C)\big/
 \bigl((p,z)\sim(pt,\chi(t)z)\bigr).
\]
The character has differential $d\chi:\Lie(T_M)\to\C$. The induced connection on $L_\chi$ has curvature $-d\chi(\pi_Y^*\overline F_A)$, so
\[
 c_1(L_\chi)
 =\left[\frac{d\chi(\pi_Y^*\overline F_A)}{2\pi i}\right].
\]
Under $X_*(T_M)\otimes\C\cong\Lie(T_M)$, its complexification is
\[
 c(\mathscr P)
 \in H^2(Y,\C)\otimes\Lie(T_M),
 \qquad
 c(\mathscr P)=\left[\frac{\pi_Y^*\overline F_A}{2\pi i}\right].
\]
Thus pairing with $B_{\gD}(m,-)$ gives
\begin{equation}\label{eq:periodchern}
 [\Omega_m]-[\Omega_{m_0}]
 =-2\pi i\,
 \bigl\langle B_{\gD}(m-m_0,-),c(\mathscr P)\bigr\rangle.
\end{equation}

\subsection{Periods of the local residue family}\label{subsec:localperiods}

Starting from $T^*\GD$, reduction by $N$ at zero gives $T^*(\GD/N)$. The principal $M$-bundle $\GD/N\to\GD/P_M$ then gives the reduction at $m+\cO$, which is $\widetilde\cJ_m$. A connection identifies these fibers with
\[
 Y_{\mathrm{loc}}=T^*(\GD/P_M)
 \times_{\GD/P_M}(\GD\times_{P_M}\cO).
\]
The $P_M$-action on $\cO$ factors through $M$. The principal $T_M$-bundle $(\GD/N\times\cO)/[M,M]\to\GD\times_{P_M}\cO$ pulls back to $Y_{\mathrm{loc}}$; its complexified characteristic class is denoted by $c_{\mathrm{loc}}$.

\begin{corollary}[Local periods]\label{cor:localperiods}
The residue fibers have holomorphic symplectic forms $\omega_m$ satisfying
\begin{equation}\label{eq:localperiodvariation}
 [\omega_m]=-2\pi i\,
 \langle B_{\gD}(m,-),c_{\mathrm{loc}}\rangle.
\end{equation}
For regular $m$, $\nu_m:\widetilde\cJ_m\to\GD\cdot(m+e)$ is symplectic.
\end{corollary}
\begin{proof}
The covector splitting and reduction-form calculation of this section apply to any principal $M$-bundle. Thus they give the smooth symplectic fibers and \eqref{eq:periodchern}; the scaling argument of Lemma~\ref{lem:exactnilpotent} gives $[\omega_0]=0$.

The symplectic comparison in the proof of Theorem~\ref{thm:simultaneous} applies to the principal $\GD$-bundle $\GD\to\Spec\C$. For regular $m$, it identifies the reduced form with the KKS form on $\GD\cdot(m+e)$, so $\nu_m$ is symplectic.
\end{proof}

\subsection{Comparison with the induced-orbit reduction}

The differentiable identifications $\psi_m:Y\to\widetilde\cL_m$ determine the open subset
\[
 U=\psi_0^{-1}(\widetilde\cL_0^\circ)\subset Y
\]
which carries the algebraic structure transported from $\widetilde\cL_0^\circ$ to $U$. The composition of $\psi_0|_U$ with the zero-fiber map of Theorem~\ref{thm:simultaneous} gives $f_0:U\to\cL_{\mathscr O_0}$. For regular $m$, the composition of $\psi_m|_U$ with the regular-fiber isomorphism gives $j_m$.

\begin{corollary}[Comparison with the induced-orbit reduction]
\label{cor:inducedperiod}
For every $m\in\zM^{\reg}$ and $e\in\cO$, there is a differentiable open embedding
\[
 j_m:U\longrightarrow\cL_{\GD\cdot(m+e)}
\]
and, in $H^2(U,\C)$,
\begin{equation}\label{eq:inducedperiod}
 [j_m^*\Omega_{\GD\cdot(m+e)}]
 -f_0^*[\Omega_{\mathscr O_0}]
 =-\left[\pi_Y^*\langle B_{\gD}(m,-),F_A\rangle|_U\right].
\end{equation}
The map $f_0$ is finite \'etale and symplectic, of generalized Springer degree. In degree one it identifies $U$ with the induced-orbit reduction.
\end{corollary}

\begin{proof}
Symplecticity in Theorem~\ref{thm:simultaneous} gives
\[
 f_0^*\Omega_{\mathscr O_0}=\Omega_0|_U
\]
and, for regular $m$,
\[
 j_m^*\Omega_{\GD\cdot(m+e)}=\Omega_m|_U.
\]
Subtracting these identities and applying \eqref{eq:formvariation} at $m_0=0$ therefore gives the stated equality after taking classes on $U$.
\end{proof}

\section{Residue Families and Examples}\label{sec:centralexamples}\label{sec:examples}

In this final section, we present more explicitly residue families in certain cases. In the examples that follow, there is always considered one hyperspecial marked point, therefore it is $\gD=\g$. Moreover, $\cB$, $(P_M,\cO)$, and $\mathscr O_0=\Ind_{\lalg}^{\gD}(\cO)$ will be as in the preceding sections. The case of several marked points is specially treated in Section~\ref{rem:multipoint}.

\subsection{Dimension and period conditions}

\begin{corollary}[Dimension and periods]\label{thm:massconditions}
For $m\in\zM^{\reg}$ and $e\in\cO$, the residue orbit is $\mathscr O_m=\GD\cdot(m+e)$. Then
\[
 \dim\mathscr O_m=\dim\mathscr O_0
 =\dim\gD-\dim\lalg+\dim\cO,
 \qquad
 \dim\widetilde\cL_m=2\dim\cB+\dim\mathscr O_0.
\]
On the manifold $Y$ of \eqref{eq:Ymodel}, we have
\[
 [\Omega_m]=-2\pi i\,
 \langle B_{\gD}(m,-),c(\mathscr P)\rangle.
\]
Thus the family of Theorem~\ref{thm:simultaneous} is mass-like in the sense of Definition~\ref{def:linearperiods}. Its comparison with $\cL_{\mathscr O_0}$ is given by the finite cover $f_0$ and the pullback formula \eqref{eq:inducedperiod}.
\end{corollary}

\begin{proof}
The dimensions follow from Propositions~\ref{prop:centralizer}, \ref{prop:induceddimension}, and~\ref{prop:smoothfamily}. The period formula is \eqref{eq:periodchern} with $m_0=0$, using Lemma~\ref{lem:exactnilpotent}. Corollary~\ref{cor:inducedperiod} finally gives the comparison.
\end{proof}

\subsection{Orbit and leaf dimensions}

For a regular residue $m+e$ associated with $(\lalg,\cO)$, the two orbit dimensions in \cite[Section~6]{BD} are
\begin{align}
 d_{\rm ss}&:=\dim\GD\cdot m=\dim\gD-\dim\lalg,
 \label{eq:BDssdimension}\\
 d&:=\dim\GD\cdot(m+e)=d_{\rm ss}+\dim\cO
 =\dim\Ind_{\lalg}^{\gD}(\cO).
 \label{eq:BDtotaldimension}
\end{align}
For a sheet with this dense Jordan class, its dimension as a variety is
\begin{equation}\label{eq:sheetversusorbitdimension}
 \dim S=d+\dim\zM,
\end{equation}
because $\zM^{\reg}/Q_M$ parametrizes the orbits. Globally, the corresponding leaf dimension is
\begin{equation}\label{eq:BDgloballeafdimension}
 \dim\cL_{\mathscr O_m}=2\dim\cB+d.
\end{equation}
In the table that follows, $r$ is the sheet rank and the Levi type is the type of $[\lalg,\lalg]$. The orbit $\cO_{\min}^{D_4}=\mathcal O_{[2^2,1^4]}$ is the nonzero nilpotent orbit of smallest dimension in $\mathfrak{so}_8$.

\vspace{3mm}

\begin{center}
\small
\renewcommand{\arraystretch}{1.15}
\begin{tabular}{c|c|c|c|c|c|c}
type & Levi type & $\cO$ & $\mathscr O_0$ & $r$ & $d_{\rm ss}$ & $d$ \\\hline
$A_6$ & $A_5$ & $0$ & $[2,1^5]$ & $1$ & $12$ & $12$ \\
$A_6$ & $A_2+2A_1$ & $0$ & $[3^2,1]$ & $2$ & $32$ & $32$ \\
$D_5$ & $D_4$ & $[2^2,1^4]$ & $[3^2,1^4]$ & $1$ & $16$ & $26$ \\
$D_5$ & $A_1+D_3$ & $0$ & $[3^2,1^4]$ & $1$ & $26$ & $26$ \\
$E_8$ & $A_3+A_2+A_1$ & $0$ & $E_8(b_6)$ & $2$ & $220$ & $220$ \\
$E_8$ & $D_4+A_2$ & $\cO_{\min}^{D_4}\times0$ & $E_8(b_6)$ & $2$ & $210$ & $220$
\end{tabular}
\end{center}

\vspace{3mm}

For a curve $X$ with genus $g_X$, the tangent space at $\mathcal E$ is $H^1(X,\mathcal E(\g))$, while regular stability gives $H^0=0$, and $\deg\mathcal E(\g)=0$. Riemann--Roch therefore gives
\begin{equation}\label{eq:hyperspecialbasedimension}
 \dim\cB=(g_X-1)\dim G.
\end{equation}

\subsection{Type \texorpdfstring{$A$}{A}: block sizes and periods}

For $\gD=\mathfrak{sl}_n$ and $n_1+\cdots+n_r=n$, the block Levi in the $\mathrm{SL}_n$ cover of $\mathrm{PGL}_n$ is
\[
 M(n_1,\ldots,n_r)=S(\mathrm{GL}_{n_1}\times\cdots\times \mathrm{GL}_{n_r})
 =\{(A_i):\textstyle\prod_i\det A_i=1\}.
\]
Its Lie algebra satisfies $\sum_i\operatorname{tr}X_i=0$, has dimension $\sum_i n_i^2-1$, and has center dimension $r-1$.

Every sheet in type $A$ is Dixmier, so its nilpotent orbit in the Levi is zero. A central parameter is the block-scalar matrix
\[
 m=\operatorname{diag}(a_1I_{n_1},\ldots,a_rI_{n_r}),
 \qquad\sum_i n_i a_i=0.
\]
Regularity means that the $a_i$ are pairwise distinct. Then $d=n^2-\sum_i n_i^2$, and induction gives $\mathcal O_{\nu^{\mathsf t}}$, where $\nu$ is the decreasing rearrangement of $(n_i)$. The group $Q_M$ permutes equal-sized blocks.

In $\mathfrak{sl}_7$, block sizes $(6,1)$ and $(3,2,2)$ give the induced partitions $[2,1^5]$ and $[3^2,1]$, with orbit dimensions $49-37=12$ and $49-17=32$. Their center dimensions are one and two. The first datum has $Q_M=1$; for the second, $Q_M\simeq\mathfrak S_2$ exchanges the two size-two blocks.

For the $(2,2)$ Levi in $\mathrm{SL}_4$, Example~\ref{ex:SL4running} gives $d=8$ and one central parameter. The sheet dimension is equal to 9, while the global leaf dimension is $30(g_X-1)+8$. More generally, the leaf dimension for the block sizes $(n_i)$ is 
$$2(g_X-1)(n^2-1)+n^2-\sum_i n_i^2.$$

Since the nilpotent orbit is zero, the manifold $Y$ in \eqref{eq:Ymodel} is $T^*\cBP$. The variation formula becomes
\[
 \Omega_m=\omega_{\mathrm{can}}
 -\pi^*\langle B_{\gD}(m,-),F_A\rangle,
\]
where $\pi:T^*\cBP\to\cBP$ is the cotangent projection, $\omega_{\mathrm{can}}$ is its canonical symplectic form, and $F_A$ is the curvature of the connection used in Section~\ref{sec:periods}.

\subsection{Two type-\texorpdfstring{$D_5$}{D5} families with the same induced orbit}
\label{subsec:D5mass}

Here $\gD=\mathfrak{so}_{10}$ has dimension $45$, and the orbit notation refers to $\mathrm{SO}_{10}$. A subspace is \emph{isotropic} if the defining symmetric form vanishes on it.

The first family uses the parabolic stabilizing an isotropic line and the datum
\[
 \lalg=\mathfrak{gl}_1\oplus\mathfrak{so}_8,
 \qquad\cO=\mathcal O_{[2^2,1^4]}\subset\mathfrak{so}_8.
\]
Here $\dim\lalg=29$. Using the transpose $[6,2]$ and four odd parts in \eqref{eq:typeDorbitdimension} gives 
$$\dim\cO=28-\tfrac12(6^2+2^2-4)=10,$$ so $d_{\rm ss}=16$ and $d=26$.

For this Levi, orthogonal induction adds two to the largest part and then takes $D$-collapse:
\[
 [2,2,1^4]\longmapsto[4,2,1^4]\longmapsto[3,3,1^4].
\]
The induced partition has transpose $[6,2,2]$ and six odd parts, giving dimension 
$$45-\tfrac12(6^2+2^2+2^2-6)=26.$$

The central coordinate $a$ has $Q_M\simeq\mathbb Z/2$, acting by $a\mapsto-a$. The involution is realized in $\mathrm{SO}_{10}$ by exchanging the two isotropic lines in their nondegenerate two-dimensional span and applying an orthogonal reflection on its eight-dimensional complement. The two determinant signs cancel. The resulting conjugation preserves the unique $\mathrm{SO}_8$-orbit with partition $[2^2,1^4]$. A normalizer can only preserve or exchange the two nonzero eigenspaces, so these are all parameter conjugacies.

The second family uses the parabolic stabilizing an isotropic two-plane and the datum
\[
 \lalg'=\mathfrak{gl}_2\oplus\mathfrak{so}_6,
 \qquad\cO'=0.
\]
The Levi dimension is $4+15=19$, so $d_{\rm ss}=d=45-19=26$. Induction gives the same Richardson orbit $\mathcal O_{[3^2,1^4]}$. Exchanging the two isotropic two-planes gives $Q_{M'}\simeq\mathbb Z/2$ on its one-dimensional center.

The central coordinates $a,b$ are normalized so that the first semisimple element has eigenvalues $(a,0^8,-a)$ and the second has eigenvalues $(b,b,0^6,-b,-b)$. The nilpotent element of the first family acts on its eight-dimensional zero-eigenspace. Hence their characteristic polynomials are respectively
\[
 t^8(t^2-a^2),\qquad t^6(t^2-b^2)^2.
\]
If these are equal, cancellation of $t^6$ and comparison of constant coefficients give $b^4=0$. Thus $b=0$, and comparison of the remaining coefficients gives $a=0$. Consequently, their images in the adjoint quotient $\gD\sslash\GD:=\Spec\C[\gD]^{\GD}$ are distinct curves intersecting exactly at the value of zero.

\subsubsection*{Counting compatible parabolic reductions}

The vector space $V$ has dimension 10 over $\C$ and a nondegenerate symmetric form $(\ ,\ )$, and $x\in\mathfrak{so}(V)$ has Jordan type $[3^2,1^4]$. The two size-three blocks give $x^3=0$, $\operatorname{rank}x=4$, and $\operatorname{rank}x^2=2$. We write $E=\operatorname{im}x^2$. Skew-adjointness of $x$ makes $x^2$ self-adjoint, and $x^4=0$ shows that $E$ is an isotropic two-plane for the original form.

A second symmetric form on $E$ is given by
\[
 b(x^2v,x^2w)=(v,x^2w).
\]
It is well-defined because $\ker x^2=E^\perp$. It is nondegenerate: a vector $x^2v$ pairing to zero with every $x^2w$ satisfies $v\in E^\perp=\ker x^2$, so $x^2v=0$. The notation $\perp_b$ denotes orthogonality in $E$ with respect to $b$.

\begin{lemma}[Compatible lines]\label{lem:D5lines}
The reductions for the isotropic-line parabolic are exactly the two $b$-isotropic lines in $E$.
\end{lemma}
\begin{proof}
For the first parabolic, a reduction above $x$ is an isotropic line $\ell$ preserved by $x$, with zero residue on $\ell$, such that the induced endomorphism of $\ell^\perp/\ell$ has Jordan type $[2^2,1^4]$. Thus $x\ell=0$, and the induced endomorphism must have square zero. The latter condition is $x^2(\ell^\perp)\subset\ell$. A rank-two linear map restricted to a hyperplane has rank at least one; therefore this inclusion forces $\ell\subset E$.

For $\ell\subset E$, the identity $b(x^2v,w)=(v,w)$ for $w\in E$ gives $x^2(\ell^\perp)\subset\ell^{\perp_b}$. Since $\ker x^2=E^\perp\subset\ell^\perp$, the left side has dimension $9-8=1$, and the right side is a line. Hence they are equal. Thus the square-zero condition is equivalent to $\ell=\ell^{\perp_b}$, or $b$-isotropy.

Conversely, a $b$-isotropic line in $E$ satisfies $x\ell=0$. Since $\ker x\subset E^\perp\subset\ell^\perp$, the restriction of $x$ to $\ell^\perp$ has rank $9-6=3$. On each size-three Jordan block, $x$ maps $\ker x^2$ onto $\operatorname{im}x^2$; the size-one blocks contribute zero. Hence $x(\ker x^2)=E$. Since $\ker x^2\subset\ell^\perp$, we have $\ell\subset x(\ell^\perp)$. Quotienting by $\ell$ lowers the rank to two. A square-zero endomorphism of the eight-dimensional space $\ell^\perp/\ell$ with rank two has two size-two Jordan blocks and four size-one blocks, giving type $[2^2,1^4]$.

A nondegenerate symmetric form on a two-dimensional complex space has exactly two isotropic lines. There are therefore two reductions.
\end{proof}

\begin{lemma}[The compatible two-plane]\label{lem:D5plane}
The isotropic-two-plane parabolic has exactly one compatible reduction, $E=\operatorname{im}x^2$.
\end{lemma}
\begin{proof}
For the second parabolic, a reduction is an isotropic two-plane $W$ with zero residue on $W$ and on $W^\perp/W$. These conditions are $xW=0$ and $x(W^\perp)\subset W$. The first, together with skew-adjointness, implies $xV\subset W^\perp$. Applying $x$ once more gives $E=x^2V\subset W$, so $W=E$. This plane does satisfy the conditions: $xE=0$, $E^\perp=\ker x^2$, and $x(\ker x^2)=E$. Thus there is exactly one reduction.
\end{proof}

\begin{proposition}[The two type-$D_5$ degrees]\label{prop:D5degrees}
For the parabolic stabilizing an isotropic line, with Levi $\mathfrak{gl}_1\oplus\mathfrak{so}_8$ and nilpotent orbit $\mathcal O_{[2^2,1^4]}$, the generalized Springer degree is two. For the parabolic stabilizing an isotropic two-plane, with Levi $\mathfrak{gl}_2\oplus\mathfrak{so}_6$ and zero nilpotent orbit, it is one. Thus the first open zero fiber gives a double cover of $\cL_{\mathcal O_{[3^2,1^4]}}$, and the second gives an isomorphism onto that leaf.
\end{proposition}

\begin{proof}
Lemmas~\ref{lem:D5lines} and~\ref{lem:D5plane} count the compatible reductions. Lemma~\ref{lem:inducedorbitcover} makes both maps over the induced orbit finite \'etale. The two counts therefore give their degrees. Theorem~\ref{thm:simultaneous} transports them to the global reductions.
\end{proof}

\subsection{Two families through \texorpdfstring{$E_8(b_6)$}{E8(b6)}}\label{subsec:E8families}

In $\mathfrak e_8$, of dimension $248$, the orbit $E_8(b_6)$ has dimension $220$. This orbit is induced by the following two Levi data, also listed in \cite[Section~6.3.3]{BD}.

The first datum consists of a Levi $\lalg_1$ of derived type $A_3+A_2+A_1$ and $\cO_1=0$. Its center has dimension two and $\dim\lalg_1=15+8+3+2=28$, so $d_{\rm ss}=d=220$.

The second datum has a Levi $\lalg_2$ of derived type $D_4+A_2$ and nilpotent orbit
\[
 \cO_2=\mathcal O_{\min}^{D_4}\times\{0\}_{A_2},
 \qquad\mathcal O_{\min}^{D_4}=\mathcal O_{[2^2,1^4]}.
\]
The minimal $D_4$ orbit is rigid of dimension $10$. The center again has dimension two, and $\dim\lalg_2=28+8+2=38$, giving $d_{\rm ss}=210$ and $d=220$.

The exceptional induction table from \cite[Section~4, Table ``Induced orbits in type $E_8$,'' row $E_8(b_6)$]{dGE} gives
\[
 \Ind_{\lalg_1}^{\mathfrak e_8}(0)
 =\Ind_{\lalg_2}^{\mathfrak e_8}
 (\mathcal O_{\min}^{D_4}\times0)=E_8(b_6).
\]
Consequently, both families have global leaf dimension $496(g_X-1)+220$ and a two-dimensional central parameter space. Their regular orbit closures are locally products with $\{0\}$ and $\overline{\mathcal O}_{\min}^{D_4}$, respectively, by Proposition~\ref{prop:residualsingularities}.

\appendix

\section{Parahoric Lattices and Several Marked Points}\label{sec:parahoric-formulas}
\subsection{The local parahoric group and its reductive quotient}

As reviewed in Section \ref{sec:level-leaves}, the local construction of a parahoric group scheme involves a smooth complex projective curve $X$, a divisor $D=\{x\}$, and a connected complex semisimple group $G$ with Lie algebra $\g$. The parahoric data consist of a maximal torus $T\subset G$, with Lie algebra $\talg=\Lie(T)$, and a rational weight $\theta\in X_*(T)\otimes_{\mathbb Z}\mathbb Q$, where $X_*(T)=\operatorname{Hom}(\mathbb G_m,T)$ is the cocharacter lattice. Here $\alpha(\theta)$ is the root--cocharacter pairing. The affine root hyperplanes $\alpha(\theta)=k$, $k\in\mathbb Z$, cut the real cocharacter space into \emph{facets}, including their faces. The facet containing $\theta$ determines a parahoric group scheme $\mathcal G_\theta$, constant with fiber $G$ away from $x$ \cite[Section~2.1]{KSZ}.

For a coordinate $z$ at $x$, its group on the formal disc is $G_\theta\subset G(\C((z)))$. The root system is denoted by $\Phi$, with root spaces $\g_\alpha$. An $\alpha$-root coefficient $z^k$ is allowed when $k+\alpha(\theta)\geq0$. Thus the Lie lattice is
\begin{equation}\label{eq:parahoricroots}
 \mathfrak g_\theta:=\Lie(G_\theta)
 =\talg[[z]]\oplus
 \bigoplus_{\alpha\in\Phi}
 z^{\lceil-\alpha(\theta)\rceil}\g_\alpha[[z]].
\end{equation}
This \emph{parahoric lattice} is a free $\C[[z]]$-module of rank $\dim G$ spanning $\g((z))$.

The pro-unipotent radical $G_\theta^+$ is the kernel of reduction to the connected reductive quotient
\[
 L_\theta=G_\theta/G_\theta^+.
\]
The parahoric root exponent is $n_\alpha=\lceil-\alpha(\theta)\rceil$. The root exponent in $\mathfrak g_\theta^+:=\Lie(G_\theta^+)$ is
\[
 n_\alpha^+=
 \begin{cases}
 n_\alpha+1,&\alpha(\theta)\in\mathbb Z,\\
 n_\alpha,&\alpha(\theta)\notin\mathbb Z,
 \end{cases}
\]
and its toral summand is $z\talg[[z]]$. Thus the quotient Lie algebra is the pseudo-Levi algebra
\[
 \Lie(L_\theta)
 \cong\Lie Z_G\bigl(\exp(2\pi i\theta)\bigr)\subset\g.
\]
Its roots satisfy $\alpha(\theta)\in\mathbb Z$. A \emph{pseudo-Levi algebra} is the Lie algebra centralizer of a semisimple group element.

Since $G$ is assumed semisimple, $Z(G)$ is finite. The effective group of level changes is
\[
 \GD=L_\theta/Z(G),\qquad
 \gD=\Lie(\GD)=\Lie(L_\theta).
\]
This connected reductive \emph{level group} has the same adjoint and coadjoint actions as $L_\theta$. The reduction map
\[
 \rho_\theta:\mathfrak g_\theta\longrightarrow\gD
\]
has kernel $\mathfrak g_\theta^+$ and is equivariant for $G_\theta\to\GD$.

We fix a nondegenerate $G$-invariant symmetric form $B_\g$. Its restriction $B_{\gD}$ identifies $\gD\cong\gD^*$: it pairs the toral summand nondegenerately with itself and each root space with its opposite. The notation
\[
 \mathfrak c_D=\gD^*\sslash\GD,
 \qquad\chi_D:\gD^*\longrightarrow\mathfrak c_D
\]
denotes the affine quotient by invariant polynomial functions.

When $|\alpha(\theta)|<1$ for every root, the roots with $\alpha(\theta)\geq0$ define a parabolic $Q_0=M_0U_0$ containing $T$. Its Levi factor $M_0$ has roots with $\alpha(\theta)=0$, and its unipotent radical $U_0$ has roots with $\alpha(\theta)>0$. Their Lie algebras are $\mathfrak q_0$, $\mathfrak m_0$, and $\mathfrak u_0$. The root exponents in \eqref{eq:parahoricroots} are zero or one, giving
\begin{equation}\label{eq:nearorigin}
 \mathfrak g_\theta=\mathfrak q_0+z\g[[z]],
 \qquad
 \mathfrak g_\theta^+=\mathfrak u_0+z\g[[z]].
\end{equation}
Thus $\rho_\theta(A_0+zA_1+\cdots)$ is the image of $A_0$ in $\mathfrak q_0/\mathfrak u_0\cong\gD$. For arbitrary weights, reduction uses the shifted root coefficients in \eqref{eq:parahoricroots}.

\begin{example}[A block calculation of the Levi residue]\label{ex:SL3leviresidue}
For the standard $(2,1)$ parabolic of $\mathrm{SL}_3$, an allowed constant coefficient in \eqref{eq:nearorigin} is
\[
 A_0=\begin{pmatrix}A&b\\0&c\end{pmatrix},
 \qquad A\in\operatorname{Mat}_{2\times2}(\C),\quad b\in\C^2,
 \quad\operatorname{tr}(A)+c=0.
\]
Here $\mathfrak u_0$ consists of the upper-right blocks, $\mathfrak m_0=\mathfrak s(\mathfrak{gl}_2\oplus\mathfrak{gl}_1)$, and
\[
 \rho_\theta(A_0+zA_1+\cdots)
 =\begin{pmatrix}A&0\\0&c\end{pmatrix}.
\]
The Higgs field $(A_0+zA_1+\cdots)\,dz/z$ has matrix residue $A_0$ and Levi residue $\operatorname{diag}(A,c)$.
\end{example}

\begin{lemma}[Duality of the parahoric lattices]
\label{lem:parahoricpairing}
The invariant form gives a perfect $\C[[z]]$-pairing
\[
 (u,v)\longmapsto z^{-1}B_\g(u,v):
 \mathfrak g_\theta^+\times\mathfrak g_\theta\longrightarrow\C[[z]].
\]
Consequently,
\begin{equation}\label{eq:leveldual}
 \mathcal E(\g)_\eta^\vee\otimes K_X
 \cong\mathcal E(\g)\otimes K_X(D).
\end{equation}
\end{lemma}

\begin{proof}
The exponents satisfy $n_\alpha^++n_{-\alpha}=1$: for integral $\alpha(\theta)$, $n_\alpha=-\alpha(\theta)$ and $n_\alpha^+=n_\alpha+1$; otherwise the two ceilings sum to one and $n_\alpha^+=n_\alpha$. Hence $z^{-1}B_\g$ pairs opposite root lattice generators by units. It also pairs $z\talg[[z]]$ perfectly with $\talg[[z]]$, while nonopposite root spaces pair to zero.

At $x$ this identifies the dual of the level-preserving sheaf with $z^{-1}\mathcal E(\g)$. Tensoring with $dz$ gives $\mathcal E(\g)\,dz/z$. Thus frame invariance and the usual adjoint duality away from $x$ give \eqref{eq:leveldual}.
\end{proof}

\subsection{The smaller parahoric subgroup}

\begin{remark}[Residues for the smaller parahoric]\label{rem:parahoricembedding}
The inverse image of $P_M$ under $G_\theta\to L_\theta\to\GD$ is the formal-disc group of a smaller parahoric $\mathcal G_{\theta'}$ \cite[Section~8.b]{PR}, with Lie lattice and effective reductive quotient
\[
 \mathfrak g_{\theta'}=\rho_\theta^{-1}(\palg)\subset\mathfrak g_\theta,
 \qquad M.
\]
Thus $\cBP=\cU/P_M$ parametrizes these parabolic reductions over $\cB$, and the fiber at $m$ has residue condition
\[
 \operatorname{pr}_{\lalg}\Res_x(\varphi)\in m+\cO,
 \quad\text{equivalently}\quad
 \Res_x(\varphi)\in m+\cO+\nalg.
\]
For a hyperspecial parahoric, the inclusion is $\palg+z\g[[z]]\subset\g[[z]]$.
\end{remark}
\subsection{Several marked points}\label{rem:multipoint}

A reduced divisor $D=\{x_1,\ldots,x_s\}$ carries fixed weights $\theta_i$ and levels, with local reductive Lie algebras $\mathfrak g_{D,i}=\Lie(L_{\theta_i})$. The levels on a fixed torsor $\mathcal E$ form a $\prod_iL_{\theta_i}$-torsor. Global central automorphisms act by the same element at every marked point. Therefore, for semisimple $G$, the effective level group and its Lie algebra are
\[
 \GD=\left(\prod_iL_{\theta_i}\right)/\Delta Z(G),
 \qquad
 \Lie(\GD)=\bigoplus_i\mathfrak g_{D,i},
\]
where $\Delta Z(G)$ is the diagonal image. Since $Z(G)$ is finite, it does not change the Lie algebra. The deformation sequence has quotient $\bigoplus_i i_{x_i*}\mathfrak g_{D,i}$; its dual gives the coresidue moment map
\[
 \boldsymbol\mu=(\mu_1,\ldots,\mu_s).
\]
For coadjoint orbits $\mathscr O_i\subset\mathfrak g_{D,i}^*$, the product orbit $\boldsymbol{\mathscr O}=\prod_i\mathscr O_i$ gives
\[
 \cL_{\boldsymbol{\mathscr O}}
 =\boldsymbol\mu^{-1}\!\left(\prod_i\mathscr O_i\right)/\GD,
 \qquad
 \dim\cL_{\boldsymbol{\mathscr O}}
 =2\dim\cB+\sum_i\dim\mathscr O_i.
\]
On a chart trivializing the level bundle,
\[
 \cP|_V\cong T^*V\times\bigoplus_i\mathfrak g_{D,i}^*,
 \qquad
 \cL_{\boldsymbol{\mathscr O}}|_V
 \cong T^*V\times\prod_i\mathscr O_i.
\]
Consequently, the components of total rank $2\dim\cB+d$ come from products of local sheets $S_i$ of orbit dimensions $d_i$ with $\sum_i d_i=d$.

For reductive $G$ with positive-dimensional center, $\Lie(\GD)=(\bigoplus_i\mathfrak g_{D,i})/\Delta\zalg(\g)$. Its dual consists of tuples annihilating the diagonal center:
\[
 \sum_i\xi_i|_{\zalg(\g)}=0.
\]
This is the residue theorem for the central Higgs component. Local sheet products are intersected with this subspace before reduction.

\begin{example}[Two hyperspecial vertices in $\mathrm{SL}_2$]\label{ex:twopointsl2}
For $G=\mathrm{SL}_2$, the example with two hyperspecial parahorics has the following local data: their local groups are $G[[z]]$ and reductive quotients are $G$. Thus $\Lie(\GD)=\mathfrak{sl}_2\oplus\mathfrak{sl}_2$. Every nonzero $\mathfrak{sl}_2$-orbit has dimension two, so $S_{\reg}=\mathfrak{sl}_2\setminus\{0\}$. The products $S_{\reg}\times\{0\}$ and $\{0\}\times S_{\reg}$ give distinct components of rank $2\dim\cB+2$, while $S_{\reg}\times S_{\reg}$ has rank $2\dim\cB+4$.

Their leaves on a chart trivializing the level bundle are
\[
 T^*V\times\mathscr O\times\{0\},\qquad
 T^*V\times\{0\}\times\mathscr O,\qquad
 T^*V\times\mathscr O_1\times\mathscr O_2.
\]
\end{example}
\subsection{Periods with several marked points}\label{rem:multipointperiod}
For semisimple $G$, the local data are parabolics $P_i=M_iN_i\subset L_{\theta_i}$ and nilpotent $M_i$-orbits $\cO_i$. Since each $M_i$ contains the image of $Z(G)$, the effective groups are
\[
 P_D=\left(\prod_iP_i\right)/\Delta Z(G),\qquad
 M_D=\left(\prod_iM_i\right)/\Delta Z(G),
\]
where $\Delta Z(G)$ is the diagonal image. The quotient spaces are $\cV_D=\cU/\prod_iN_i$ and $\mathcal B_{P_D}=\cU/P_D$; then $\cV_D\to\mathcal B_{P_D}$ is a principal $M_D$-bundle. For $\mathfrak z_i=\mathfrak z(\Lie(M_i))$ and $T_{M_i}=M_i/(M_i)_{\der}$, its abelian torus is
\[
 T_D=M_D/(M_D)_{\der}
 \cong\left(\prod_iT_{M_i}\right)/C_D,
 \qquad
 C_D=\operatorname{im}\left(\Delta Z(G)\longrightarrow\prod_iT_{M_i}\right).
\]
Here $(M_D)_{\der}$ is the image of $\prod_i(M_i)_{\der}$. With character lattice $X^*(T)=\operatorname{Hom}(T,\mathbb G_m)$, the integral lattices are
\[
 X^*(T_D)
 =\left\{\chi\in\bigoplus_iX^*(T_{M_i}):\chi|_{C_D}=1\right\},
 \qquad
 X_*(T_D)=\operatorname{Hom}_{\mathbb Z}(X^*(T_D),\mathbb Z).
\]
Under $\Lie(T_D)\cong\bigoplus_i\mathfrak z_i$, the lattice $\bigoplus_iX_*(T_{M_i})$ has index $|C_D|$ in $X_*(T_D)$.

The fibers are differentiably identified with
\[
 Y_D=T^*\mathcal B_{P_D}\times_{\mathcal B_{P_D}}
       \left(\cV_D\times_{M_D}\prod_i\cO_i\right).
\]
The quotient $(\cV_D\times\prod_i\cO_i)/(M_D)_{\der}$ pulls back to a $T_D$-bundle $\mathscr P_D\to Y_D$ with integral class in $H^2(Y_D,X_*(T_D))$. Complexifying and using the pairing in \eqref{eq:periodchern} gives classes $c_i$. For $\mathbf m=(m_i)$ and $\mathbf m_0=(m_{0,i})$,
\[
 [\Omega_{\mathbf m}]-[\Omega_{\mathbf m_0}]
 =\sum_i\langle m_i-m_{0,i},c_i\rangle,
 \qquad
 c_i\in H^2(Y_D,\C)\otimes\mathfrak z_i^*.
\]
For $\mathrm{SL}_2$ with two hyperspecial points and diagonal torus Levis, $T_D=(\mathbb G_m^2)/\{(1,1),(-1,-1)\}$, and
\[
 X^*(T_D)=\{(a,b)\in\mathbb Z^2:a+b\text{ is even}\},
 \qquad
 X_*(T_D)=\mathbb Z^2+\mathbb Z(\tfrac12,\tfrac12).
\]
Thus a local character pair defines an associated line bundle precisely when its exponents have even sum.

For reductive $G$, the quotient is by the full diagonal center. With a nondegenerate invariant form $B_\g$, the central parameters satisfy
\[
 \sum_i B_{\g}(m_i,-)|_{\zalg(\g)}=0.
\]
The descended tuple $B_\g(m_i,-)$ pairs with the effective torus class to give the period variation.
\section{The Hitchin Map with a Central Residue Parameter}
\label{sec:sheet-hitchin}

\subsection{Invariant polynomials and the leading polar coefficients}

The Hitchin map records invariant polynomials of the Higgs field. Their leading polar coefficients record invariant polynomials of the reductive residue.

The polynomials $p_1,\ldots,p_r$ are fixed homogeneous generators of $\C[\g]^G$, with $r=\operatorname{rank}G$ and $d_i=\deg p_i$. Invariance and homogeneity give a morphism
\[
 \cX=T^*\cU\longrightarrow\mathcal B_\theta
 =\bigoplus_{i=1}^r H^0\bigl(X,(K_X(D))^{\otimes d_i}\bigr),
 \qquad
 (\mathcal E,\varphi,\eta)\longmapsto
 (p_1(\varphi),\ldots,p_r(\varphi));
\]
local regularity follows from Lemma~\ref{lem:residualcompatibility} below. Its scheme-theoretic image is $\mathcal A_\theta\subset\mathcal B_\theta$. Since $\cX$ is smooth, $\mathcal A_\theta$ is reduced. The resulting morphism $\widetilde H_\theta$ descends along the level quotient $\ell:\cX\to\cP$ to
\[
 h_\theta:\cP\longrightarrow\mathcal A_\theta,
 \qquad \widetilde H_\theta=h_\theta\circ\ell.
\]

The residue trivialization $K_X(D)|_x\cong\C$ sends $dz/z$ to $1$. Thus $b_i=f_i(z)(dz/z)^{d_i}$ has leading polar coefficient $f_i(0)$, independent of the coordinate. Writing $\mathfrak c_G=\g\sslash G$, with coordinates $p_1,\ldots,p_r$, gives
\[
 \operatorname{ev}_x^G:\mathcal B_\theta\longrightarrow\mathfrak c_G,
 \qquad (b_1,\ldots,b_r)\longmapsto(f_1(0),\ldots,f_r(0)).
\]
The same notation denotes its restriction to $\mathcal A_\theta$.

The invariant form identifies $\gD$ and $\gD^*$, and their adjoint quotient is $\mathfrak c_D=\gD\sslash\GD$. Under the identification $\gD\simeq\Lie Z_G(\exp(2\pi i\theta))\subset\g$, it contains $\talg=\Lie(T)$. The groups $W_D$ and $W_G$ are generated by the root reflections of $\gD$ and $\g$ on $\talg$. These are their Weyl groups, with $W_D\subset W_G$. Chevalley restriction gives a finite morphism
\[
 \nu_{D,G}:\mathfrak c_D=\talg/W_D\longrightarrow
 \mathfrak c_G=\talg/W_G.
\]
Indeed, $\C[\talg]^{W_D}$ is an intermediate finite module over $\C[\talg]^{W_G}$.

\begin{lemma}[Invariant polynomials and residues]\label{lem:residualcompatibility}
The Hitchin morphism and the reductive-residue moment map satisfy
\[
 \operatorname{ev}_x^G\circ\widetilde H_\theta
 =\nu_{D,G}\circ\chi_D\circ\mu
 \quad\text{as morphisms }\cX\longrightarrow\mathfrak c_G.
\]
\end{lemma}

\begin{proof}
The Higgs field has local expression $X(z)\,dz/z$, with $X(z)\in\mathfrak g_\theta$ and coresidue $\rho_\theta(X)$. An integer $d>0$ with $d\theta$ a cocharacter gives a cover $z=t^d$ and transformed coefficient $Y(t)=\Ad(t^{d\theta})X(t^d)$. An $\alpha$-root term $z^k$ has $k\geq\lceil-\alpha(\theta)\rceil$ by \eqref{eq:parahoricroots}, and becomes a term of degree $d(k+\alpha(\theta))\geq0$. Hence $Y(t)\in\g[[t]]$. Its constant terms are the constant $\talg$-coefficient and the root terms with $\alpha(\theta)\in\mathbb Z$ and $k=-\alpha(\theta)$; these represent $\rho_\theta(X)$ in $\gD$.

For a homogeneous invariant $f$ of degree $q$, invariance gives $f(X(t^d))=f(Y(t))$. Since $f(Y(t))$ has no negative powers of $t$, $f(X(z))$ has no negative powers of $z$. Its constant coefficient is $f(Y(0))=f(\rho_\theta(X))$. Under $dz/z=d\,dt/t$, both the pullback of $f(X(z))(dz/z)^q$ and its residue generator acquire $d^q$, so the leading coefficient remains $f(\rho_\theta(X))$.

Applying this calculation to the $p_i$ proves the identity, since restriction of the same $W_G$-invariants to $\talg/W_D$ is $\nu_{D,G}$. These polynomial identities in the local coefficients hold over nonreduced bases as well.
\end{proof}

\subsection{Retaining the central residue in the target}

The Levi $M\subset\GD$ is the one fixed in the simultaneous family. The Hitchin invariants retain the conjugacy invariants of $m$; the fiber product below also retains $m$ itself. We observe that, for $e'\in\cO$, the semisimple part of $m+e'$ is $m$, so $\chi_D(m+e')=\chi_D(m)$.

\begin{definition}[The Hitchin base with a central parameter]\label{def:enhancedbase}
The \emph{Hitchin base with a central residue parameter} is defined as
\[
 \mathcal A_{\theta,M}
 =\mathcal A_\theta\times_{\mathfrak c_G}\zM,
\]
formed using $\operatorname{ev}_x^G$ and $m\mapsto\nu_{D,G}(\chi_D(m))$. Its points are pairs $(a,m)$, where $a$ is a Hitchin value and $m$ a central residue, subject to $\operatorname{ev}_x^G(a)=\nu_{D,G}(\chi_D(m))$.
\end{definition}

\begin{example}[The residue parameter for $\mathrm{SL}_2$]\label{ex:SL2enhancedbase}
For $G=\mathrm{SL}_2$ and a parahoric with diagonal-torus reductive quotient, the central parameter is $m(a)=\operatorname{diag}(a,-a)$. Then $\mathfrak c_D=\mathbb A^1_a$, whereas $p(X)=\tfrac12\operatorname{tr}(X^2)$ identifies $\mathfrak c_G=\mathbb A^1_u$ and gives $\nu_{D,G}(a)=a^2$. Hence $\mathcal A_{\theta,T}$ is defined by $\operatorname{ev}_x^G(b)=a^2$.

For $a\ne0$, the residues $a$ and $-a$ define distinct torus coadjoint orbits and residue leaves with the same ambient invariant. Here $Q_M=1$, since $M=\GD=T$; the identification $a\sim-a$ comes from the ambient $\mathrm{SL}_2$-Weyl group. The scheme fiber over $u=0$ is defined by $a^2=0$.
\end{example}

The fibers of $\chi_D|_{\zM^{\reg}}$ are $N_{\GD}(M)/M$-orbits: equal reductive invariants give conjugate semisimple elements, and their common centralizer $M$ forces the conjugating element to normalize $M$. Preserving $\cO$ restricts this group to $Q_M=N_{\GD}(M,\cO)/M$.

A value $\tau\in\mathfrak c_D$ defines the fiber $\mathcal A_\theta[\tau]=\mathcal A_\theta\times_{\mathfrak c_G}\{\tau\}$, using $\nu_{D,G}(\tau)$. If $\mathscr O\subset\gD^*$ has invariant value $\tau$, Lemma~\ref{lem:residualcompatibility} places $h_\theta(\cL_{\mathscr O})$ in this fiber. The reduced closed image of this restriction is $\mathcal A_{\theta,\mathscr O}$:
\begin{equation}\label{eq:leafwiseHitchin}
 h_{\mathscr O}:\cL_{\mathscr O}
 \longrightarrow\mathcal A_{\theta,\mathscr O}
 \hookrightarrow\mathcal A_\theta[\tau].
\end{equation}

\subsection{The Hitchin morphism on the simultaneous family}

By Lemma~\ref{lem:cotangentN}, $\mu_N^{-1}(0)/N=T^*\cV$, where $\cV=\cU/N$. A point has a faithfully flat local lift whose reductive coresidue lies in $\palg=\lalg\oplus\nalg$, with Levi component $\mu_M$. The Hitchin invariants descend to $h_P:T^*\cV\to\mathcal A_\theta$. For $v\in\mu_M^{-1}(\zM+\cO)$, the Levi residue decomposes as $\mu_M(v)=m+e'$, with $m\in\zM$ and $e'\in\cO$.

\begin{theorem}[Hitchin map on the simultaneous reduction]\label{thm:sheethitchin}
There is a canonical morphism
\[
 \widetilde h:\widetilde\cL\longrightarrow\mathcal A_{\theta,M},
 \qquad [v]\longmapsto\bigl(h_P(v),\operatorname{pr}_{\zM}\mu_M(v)\bigr).
\]
Its second component is $q$.
\end{theorem}

\begin{proof}
A local lift of $v$ to $\mu_N^{-1}(0)$ has reductive coresidue $m+e'+n$, with $n\in\nalg$. There is a rational cocharacter with weight zero on the simple roots of $M$ and weight one on the omitted simple roots. Clearing denominators gives a cocharacter $\lambda:\mathbb G_m\to Z(M)^\circ$ with positive weights on $\nalg$. Conjugation by $\lambda(t)$ fixes $\lalg$ and contracts $\nalg$ to zero. Constancy of invariant polynomials along this conjugation therefore gives
\[
 \chi_D(m+e'+n)=\chi_D(m+e')=\chi_D(m).
\]
Lemma~\ref{lem:residualcompatibility} therefore yields
\[
 \operatorname{ev}_x^G\bigl(h_P(v)\bigr)
 =\nu_{D,G}(\chi_D(m)).
\]
The polynomial identities descend from the faithfully flat local lifts, giving the morphism to the fiber product. Both components are $M$-invariant, so they descend along $\mu_M^{-1}(\zM+\cO)\to\widetilde\cL$.
\end{proof}

\begin{corollary}[Regular fibers]\label{prop:regularfibercorestriction}
For $m\in\zM^{\reg}$ and any $e\in\cO$, the residue orbit is $\mathscr O_m=\GD\cdot(m+e)$. There are canonical identifications
\[
 \widetilde\cL_m\cong\cL_{\mathscr O_m},
 \qquad
 (\mathcal A_{\theta,M})_m\cong\mathcal A_\theta[\chi_D(m)].
\]
The restriction of $\widetilde h$ is $h_{\mathscr O_m}$ followed by the inclusion of its reduced image in the fixed-invariant fiber.
\end{corollary}

\begin{proof}
The source identification is Theorem~\ref{thm:simultaneous}; the target identification follows by base change of the defining fiber product. Both morphisms apply the $p_i$ to the same Higgs field on local lifts to $\cX$, so they agree by descent.
\end{proof}

On $\widetilde\cL_0^\circ$, the zero-fiber comparison $f_0$ of Theorem~\ref{thm:simultaneous} preserves the Hitchin invariants:
\[
 \operatorname{pr}_{\mathcal A_\theta}\circ
 \widetilde h|_{\widetilde\cL_0^\circ}
 =h_\theta|_{\cL_{\mathscr O_0}}\circ f_0.
\]
\section{Local Residues and Parabolic Reductions on the Formal Disc}\label{sec:affine-sheets}

\subsection{Lattices and parabolic reductions over \texorpdfstring{$\C$}{C}}

A fixed coordinate $z$ at $x$ gives the formal Laurent series field $K_x=\C((z))$, with a chosen algebraic closure $\overline K_x$. In a fixed punctured-disc trivialization, the Higgs field has expression
\[
 \phi=\gamma\,\frac{dz}{z},\qquad \gamma\in\g(K_x).
\]
A coset $q_\theta=gG_\theta$ defines an extension with adjoint lattice $\Ad(g)\mathfrak g_\theta$. The Higgs field extends precisely when
\[
 \Ad(g^{-1})\gamma\in\mathfrak g_\theta,
\]
with reductive residue $\rho_\theta(\Ad(g^{-1})\gamma)$. This condition is independent of the representative $g$.

For $P_M=MN\subset\GD$ as in Section~\ref{sec:leviresidues}, its inverse image under $L_\theta\to\GD$ is $P_M^{L_\theta}$. The subgroup $\widehat P_M$ is the inverse image of $P_M^{L_\theta}$ under $G_\theta\to L_\theta$. Its Lie lattice is
\[
 \widehat\palg=\rho_\theta^{-1}(\palg).
\]
The coset $g\widehat P_M$ records a parahoric lattice and a $P_M$-reduction; for fixed $q_\theta$ the reductions form $\GD/P_M$. Preservation of the parabolic reduction and the central Levi residue condition are
\[
 \Ad(g^{-1})\gamma\in\widehat\palg,\qquad
 \operatorname{pr}_{\lalg}\rho_\theta(\Ad(g^{-1})\gamma)\in m+\cO,
 \qquad m\in\zM.
\]
The nilradical component and all coefficients in $\mathfrak g_\theta^+$ are unrestricted. An affine Springer fiber parametrizes parahoric lattices preserved by $\gamma$ \cite{KL,GKM}. The displayed conditions also prescribe the Levi residue \cite{Ngo}.

\subsection{The same conditions in families}

For a $\C$-algebra $R$, the loop functors are
\[
 LG(R)=G(R((z))),\qquad
 L\g(R)=\g(R((z))),\qquad
 L^+\mathcal G_\theta(R)=\mathcal G_\theta(R[[z]]),
\]
where $R((z))=R[[z]][z^{-1}]$. Reductive reduction
$\operatorname{red}_\theta:L^+\mathcal G_\theta\to L_\theta$ defines
\[
 \widehat P_M(R)
 =\operatorname{red}_\theta^{-1}\bigl(P_M^{L_\theta}(R)\bigr).
\]
Its complex points give the preceding subgroup. Since $P_M^{L_\theta}$ is parabolic, the facet--parabolic correspondence identifies $\widehat P_M$ with the positive-loop group of the corresponding parahoric \cite[\S8.b]{PR}.

An fpqc cover is a faithfully flat, quasi-compact change of base. The fpqc quotient glues cosets represented on such covers; a representative need not exist on the original base. The partial affine flag sheaf
\[
 LG/_{\mathrm{fpqc}}\widehat P_M
\]
is an ind-scheme, a filtered union of schemes along closed immersions \cite[Theorem~1.4]{PR}. For a positive-loop group $\mathcal H$ acting on a coefficient scheme $Z$, the associated fpqc sheaf is denoted by $LG\times^{\mathcal H}_{\mathrm{fpqc}}Z$.

The $k$-th jet retains the coefficient lattice modulo $z^k\mathfrak g_\theta$. These finite truncations give
\[
 \mathfrak g_\theta=\varprojlim_k\mathfrak g_{\theta,k},
 \qquad
 \mathfrak g_{\theta,k}=\mathfrak g_\theta/z^k\mathfrak g_\theta,
 \qquad
 \mathfrak g^+_{\theta,k}=\mathfrak g_\theta^+/z^k\mathfrak g_\theta,
 \qquad k\geq1.
\]
These jets are finite-dimensional affine spaces; the last quotient exists since $z\mathfrak g_\theta\subset\mathfrak g_\theta^+$. The residue factors through every jet. The root decomposition \eqref{eq:parahoricroots} gives a linear section of $\rho_\theta$, compatible with truncation. In particular, a sheet $S\subset\gD$ has inverse image $\rho_\theta^{-1}(S)\cong S\times\mathfrak g_\theta^+$.

\subsection{Local residues and their finite jets}

\begin{definition}[Local residues with a parabolic reduction]\label{def:localresiduefamily}
The coefficient scheme with prescribed Levi residue is
\[
 \widehat{\mathcal R}
 =\{(m,y)\in\zM\times\widehat\palg:
 \operatorname{pr}_{\lalg}\rho_\theta(y)\in m+\cO\}.
\]
Thus $\rho_\theta(y)\in m+\cO+\nalg$. The associated quotient sheaf is
\[
 \widetilde\cJ^{\mathrm{aff}}
 =LG\times^{\widehat P_M}_{\mathrm{fpqc}}
   \widehat{\mathcal R},
 \qquad
 (g;(m,y))\cdot p=(gp;(m,\Ad(p^{-1})y)).
\]
A point $[g;(m,y)]$ consists of the lattice with its parabolic reduction, the central parameter, and a coefficient $y$ in that lattice. Transporting $y$ to the fixed punctured-disc frame gives
\begin{equation}\label{eq:forgetlocalreduction}
 \nu^{\mathrm{aff}}:
 \widetilde\cJ^{\mathrm{aff}}
 \longrightarrow\zM\times L\g,
 \qquad
 [g;(m,y)]\longmapsto(m,\Ad(g)y).
\end{equation}
\end{definition}

The quotient $\widehat\palg_k=\widehat\palg/z^k\mathfrak g_\theta$ has induced residue map $\rho_{\theta,k}:\widehat\palg_k\to\palg$. The finite-jet scheme
\begin{equation}\label{eq:residuejets}
 \widehat{\mathcal R}_k
 =\{(m,y_k)\in\zM\times\widehat\palg_k:
   \operatorname{pr}_{\lalg}\rho_{\theta,k}(y_k)\in m+\cO\}
\end{equation}
is locally closed and of finite type, being the inverse image of the locally closed orbit $\cO$. Its transition morphisms are affine, and
\[
 \widehat{\mathcal R}
 =\varprojlim_k\widehat{\mathcal R}_k
\]
exists as a scheme, whose $R$-points are compatible allowed jets.

For the hyperspecial lattice $\mathfrak g_\theta=\g[[z]]$, the residue is the constant term and the quotient at $k=1$ is
\[
 \{(m,a)\in\zM\times\palg:
     \operatorname{pr}_{\lalg}(a)\in m+\cO\}
 =\{(m,a):a\in m+\cO+\nalg\}.
\]
Its $\GD\times^{P_M}$-space is $\widetilde\cJ$ from Section~\ref{sec:leviresidues}. For general parahorics, this condition holds in the reductive quotient, with all higher coefficients allowed.

\begin{lemma}[Complex points of the partial affine flag]
\label{lem:complexpartialflag}
Every fpqc $\widehat P_M$-torsor over $\Spec\C$ has a section. Consequently
\[
 \bigl(LG/_{\mathrm{fpqc}}\widehat P_M\bigr)(\C)
 =G(\C((z)))/\widehat P_M(\C).
\]
\end{lemma}

\begin{proof}
The group $\widehat P_{M,k}$ is the inverse image of $P_M^{L_\theta}$ under reductive reduction of the jet group $R\mapsto\mathcal G_\theta(R[z]/(z^k))$. These truncation quotients satisfy
\[
 \widehat P_M=\varprojlim_k\widehat P_{M,k}.
\]
The first quotient extends the connected parabolic $P_M^{L_\theta}$ by a connected unipotent group. Smoothness of the parahoric gives surjective transitions with kernels isomorphic to additive vector spaces, represented by the infinitesimal coefficients modulo $z^{k+1}$.

An fpqc torsor under the affine group $\widehat P_M$ is affine by descent. A torsor $\mathscr T$ has truncations $\mathscr T_k=\mathscr T\times^{\widehat P_M}\widehat P_{M,k}$. The nonempty finite-type scheme $\mathscr T_1$ has a complex point. Once $t_k\in\mathscr T_k(\C)$ is chosen, the fiber above it is a nonempty finite-type torsor under an additive vector group. It has a complex point $t_{k+1}$, so the points can be chosen compatibly. Moreover, $\mathscr T\cong\varprojlim_k\mathscr T_k$: this identity holds after fpqc trivialization and descends. Thus $(t_k)$ defines a section.

A quotient point is a torsor with an equivariant map to $LG$. A section gives $g\in G(\C((z)))$, unique up to right multiplication by $\widehat P_M(\C)$.
\end{proof}

\begin{proposition}[Finite-jet structure and fpqc fibers]
\label{prop:affinestructure}
The coefficient sheaf $\widehat{\mathcal R}$ is represented by the inverse-limit scheme of \eqref{eq:residuejets} and is $\widehat P_M$-stable. Each projection
\[
 \widehat{\mathcal R}_k\longrightarrow\zM
\]
is smooth, and the transition morphisms are smooth and affine. A linear section of $\rho_\theta$ identifies
\[
 \widehat{\mathcal R}
 \cong\zM\times\cO\times\nalg\times\mathfrak g_\theta^+.
\]
For $(m,\gamma)\in\zM(\C)\times L\g(\C)$, the fpqc fiber
\begin{equation}\label{eq:sheetASfunctor}
 \cF_{\gamma,m}
 =\widetilde\cJ^{\mathrm{aff}}
   \times_{\zM\times L\g,(m,\gamma)}\Spec\C
\end{equation}
is the subfunctor of $LG/_{\mathrm{fpqc}}\widehat P_M$ represented locally by $g\in LG(R)$ satisfying
\[
 \Ad(g^{-1})\gamma\in\widehat\palg(R),\qquad
 \operatorname{pr}_{\lalg}\rho_\theta(\Ad(g^{-1})\gamma)
 \in m+\cO(R).
\]
In particular,
\begin{equation}\label{eq:sheetAS}
 \cF_{\gamma,m}(\C)
 =\left\{
 g\widehat P_M\in G(K_x)/\widehat P_M:
 \begin{array}{l}
  \Ad(g^{-1})\gamma\in\widehat\palg,\\[-2pt]
  \operatorname{pr}_{\lalg}\rho_\theta(\Ad(g^{-1})\gamma)
     \in m+\cO
 \end{array}
 \right\}.
\end{equation}
If $\gamma$ is generically regular semisimple, so its centralizer in $G_{\overline K_x}$ is a maximal torus, this is an affine Springer fiber with the prescribed Levi residue.
\end{proposition}

\begin{proof}
The homomorphism $\widehat P_M\to P_M^{L_\theta}\to P_M$ intertwines conjugation. Stability of $m+\cO+\nalg$ under $P_M$ therefore gives stability of the coefficient scheme and makes \eqref{eq:forgetlocalreduction} well defined.

The root decomposition gives a linear splitting $s:\palg\to\widehat\palg$ of
\[
 0\longrightarrow\mathfrak g_\theta^+
 \longrightarrow\widehat\palg
 \longrightarrow\palg\longrightarrow0.
\]
Since $\palg=\lalg\oplus\nalg$, each allowed coefficient is uniquely
\[
 y=s(m+e+u)+v,\qquad
 e\in\cO,\quad u\in\nalg,\quad v\in\mathfrak g_\theta^+.
\]
Consequently, truncation gives compatible isomorphisms
\[
 \widehat{\mathcal R}_k
 \cong\zM\times\cO\times\nalg\times\mathfrak g^+_{\theta,k}.
\]
The orbit and vector-space factors are smooth. Transitions are identity maps on the first three factors and linear projections on the fourth, with kernel $z^k\mathfrak g_\theta/z^{k+1}\mathfrak g_\theta$. Taking inverse limits gives the asserted splitting of $\widehat{\mathcal R}$.

In the fiber over $(m,\gamma)$, the coefficient is $y=\Ad(g^{-1})\gamma$. Membership in $\widehat{\mathcal R}$ gives the two residue conditions, which descend by $\widehat P_M$-stability. Lemma~\ref{lem:complexpartialflag} gives the complex-point description.
\end{proof}

\subsection{Forgetting the parabolic reduction}

The parahoric affine Springer subfunctor is
\[
 \cF^{G_\theta}_\gamma(R)
 =\{q\in(LG/_{\mathrm{fpqc}}L^+\mathcal G_\theta)(R):
   \text{fpqc-locally }q=g(L^+\mathcal G_\theta),\
   \Ad(g^{-1})\gamma\in\mathfrak g_\theta(R)\}.
\]
The forgetful map is induced by
\[
 LG/_{\mathrm{fpqc}}\widehat P_M
 \longrightarrow LG/_{\mathrm{fpqc}}L^+\mathcal G_\theta.
\]
For $m\in\zM^{\reg}$ and $e\in\cO$, Theorem~\ref{thm:parabolicfibers} gives
\[
 \GD\times^{P_M}(m+\cO+\nalg)\simeq\GD\cdot(m+e).
\]
For a representative $g$, compatible parabolic reductions form the fiber of $\nu_m$ over $\rho_\theta(\Ad(g^{-1})\gamma)$. The isomorphism therefore gives a unique reduction functorially, and descends to identify $\cF_{\gamma,m}$ with the subfunctor of $\cF^{G_\theta}_\gamma$ whose residue lies in $\GD\cdot(m+e)$.

At $m=0$, over lattices with residue in $\mathscr O_0=\Ind_{\lalg}^{\gD}(\cO)$, forgetting the parabolic reduction is the base change of the finite \'{e}tale cover of Lemma~\ref{lem:inducedorbitcover}, with generalized Springer degree.

\subsection{The formal-disc condition in a Hitchin fiber}

\begin{proposition}[The fiber with a fixed central parameter]
\label{prop:localsheetfiber}
For $(a,m)\in\mathcal A_{\theta,M}$, the Hitchin fiber with central parameter $m$ is
\begin{equation}\label{eq:sheetHitchinfiber}
 \left\{v\in T^*\cV:
 h_P(v)=a,\
 \mu_M(v)\in m+\cO\right\}/M.
\end{equation}
A punctured-disc Higgs field with a fixed trivialization has coefficient $\gamma$. Its compatible parahoric extensions with $P_M$-reduction are parametrized by $\cF_{\gamma,m}(\C)$. For a fixed underlying lattice $q_\theta$, the reductions are the complex points of
\[
 \cF_{\gamma,m}
 \times_{LG/_{\mathrm{fpqc}}L^+\mathcal G_\theta}\{q_\theta\}.
\]
\end{proposition}

\begin{proof}
Base change of the $M$-torsor $\mu_M^{-1}(\zM+\cO)\to\widetilde\cL$ over $(a,m)$ imposes the two equations in \eqref{eq:sheetHitchinfiber}; quotienting gives the global fiber. Locally Proposition~\ref{prop:affinestructure} identifies the allowed extensions with $\cF_{\gamma,m}$. Fixing $q_\theta$ further imposes $gL^+\mathcal G_\theta=q_\theta$, giving the fiber product. Both descriptions are invariant under simultaneous changes of punctured-disc trivialization.
\end{proof}

\section{Quotients of Cotangent Bundles}\label{sec:cotangentquotients}

\begin{proposition}[Quotients of cotangent bundles]\label{prop:cotangentquotientcharts}
A connected complex algebraic group $H$, with $\mathfrak h=\Lie(H)$, acts on a right $H$-torsor $r:W\to Z$ in the \'{e}tale topology between smooth separated algebraic spaces of finite type. Its cotangent moment map is $J:T^*W\to\mathfrak h^*$. For every $H$-stable locally closed subspace $Y\subset\mathfrak h^*$, the geometric quotient $J^{-1}(Y)/H$ exists as a separated algebraic space of finite type. The quotient map is an $H$-torsor in the \'{e}tale topology. On every \'{e}tale trivialization $V\to Z$ of $W$ there is an isomorphism
\begin{equation}\label{eq:cotangentquotientchart}
 \bigl(J^{-1}(Y)/H\bigr)|_V\cong T^*V\times Y.
\end{equation}
In particular the quotient is smooth if $Y$ is smooth. If $f:Y\to T$ is an $H$-invariant smooth morphism and $H$ acts trivially on $T$, then $J^{-1}(Y)/H\to T$ is smooth as well.
\end{proposition}

\begin{proof}
A section $s:V\to W$ gives coordinates $w=s(v)h$ and an identification $W|_V=V\times H$. The form $h^{-1}dh$ sends the tangent vector $hA$ to $A\in\mathfrak h$, so a cotangent vector is
\[
 \alpha+\langle\xi,h^{-1}dh\rangle,
 \qquad\alpha\in T^*V,\quad h\in H,\quad\xi\in\mathfrak h^*.
\]
The infinitesimal right action is $hA$, hence $J(\alpha,h,\xi)=\xi$ and
\[
 (\alpha,h,\xi)\cdot k
 =\bigl(\alpha,hk,\Ad^*_{k^{-1}}\xi\bigr).
\]
Consequently, for coadjoint-stable $Y$, $J^{-1}(Y)=T^*V\times H\times Y$. Each orbit has the unique representative $(\alpha,1,\zeta)$ with $\zeta=\Ad_h^*\xi$, giving the quotient $T^*V\times Y$ and a trivial $H$-torsor over it.

Using the coordinate $\zeta=\Ad_h^*\xi$, the covector is $\alpha+\langle\zeta,dh\,h^{-1}\rangle$. For another section $s'=sa$, with $a:V\to H$, the same point has $h=ah'$ and
\[
 dh\,h^{-1}=(da)a^{-1}+\Ad_a(dh'\,{h'}^{-1})
\]
hence the transitions
\[
 \zeta'=\Ad_{a^{-1}}^*\zeta,
 \qquad
 \alpha'=\alpha+\langle\zeta,(da)a^{-1}\rangle.
\]
Thus the algebraic transitions preserve $Y$ and satisfy the cocycle condition. Effective \'{e}tale descent glues the quotients and torsor maps over $Z$. The charts are separated and of finite type over $V$, so these properties descend to the quotient over $Z$. Since $Z$ is separated and of finite type over $\C$, so is the quotient. Smoothness follows from the product description, as does smoothness over $T$ when $f$ is smooth.
\end{proof}

\vspace{5mm}

\textbf{Acknowledgements}.
Georgios Kydonakis is very grateful to the Kavli Institute for the Physics and Mathematics of the Universe at the University of Tokyo for its support and hospitality, where this work was partially carried out. Lutian Zhao was supported by JSPS KAKENHI Grant Number JP25K17226.
\vspace{2mm}

\textbf{AI Acknowledgement}.
The tool ChatGPT was used to assist with drafting and polishing the manuscript. The authors take full responsibility for the mathematical content and the final text.

\end{document}